\documentclass[11pt]{amsart}
\usepackage[margin=3cm]{geometry}
\usepackage{bm}
\usepackage{amsthm}
\usepackage{amssymb}
\usepackage{mathrsfs}
\usepackage{hyperref,xcolor}

\newcommand{\R}{\mathbb{R}}
\newcommand{\e}{\epsilon}
\DeclareMathOperator{\Ric}{Ric}

\DeclareMathOperator{\supp}{supp}
\DeclareMathOperator{\ind}{ind}

\begin{document}

\begin{abstract} In this article we study eternal solutions to the Allen-Cahn equation in the 3-sphere, in view of the connection between the gradient flow of the associated energy functional, and the mean curvature flow. We construct eternal integral Brakke flows that connect Clifford tori to equatorial spheres, and study a family of such flows, in particular their symmetry properties. Our approach is based on the realization of Brakke's motion by mean curvature as a singular limit of Allen-Cahn gradient flows, as studied by Ilmanen \cite{Ilmanen} and Tonegawa \cite{T03}, and it uses the classification of ancient gradient flows in spheres, by K. Choi and C. Mantoulidis \cite{CM}, as well as the rigidity of stationary solutions with low Morse index proved by F. Hiesmayr \cite{H}.
\end{abstract}

\title[Parabolic Allen-Cahn equation on the three-sphere]{Mean curvature flow and low energy solutions of the parabolic Allen-Cahn equation on the three-sphere}
\author{Jingwen Chen, Pedro Gaspar}
\address{Department of Mathematics, The University of Chicago, 5734 S University Ave, Chicago IL, 60615}
\email{jingwch@uchicago.edu, pgaspar@uchicago.edu}
\maketitle

\newtheorem{thm}{Theorem}
\newtheorem*{thm*}{Theorem}
\newtheorem{lem}{Lemma}
\newtheorem{prop}{Proposition}
\newtheorem{coro}{Corollary}

\theoremstyle{definition}
\newtheorem*{defi}{Definition}
\newtheorem*{ex}{Example}
\newtheorem*{rmk}{Remark}

\section{Introduction}
We study low energy solutions of the \emph{parabolic Allen-Cahn equation}
\begin{equation} \label{PAC}
\e \, \partial_t u = \epsilon \Delta_g u - \frac{1}{\epsilon} W'(u).\tag{PAC}
\end{equation}
in the $3$-sphere in connection with the \emph{mean curvature flow} (MCF). Here $W$ is a nonnegative double-well potential with two wells at $\pm 1$, such as $W(u)=(1-u^2)^2/4$.

The Allen-Cahn equation models phase transition and separation phenomena \cite{AC}. The motion of the diffuse transition region -- where $u$ remains bounded away from minima of $W$ -- was studied by several authors, who established the convergence of of this interface, as $\e \downarrow 0$, to a hypersurface which evolves by the mean curvature flow, in different formulations. We mention here \cite{BronsardKohn, Chen, ChenGigaGoto, EvansSonerSouganidis, MottoniSchatzman, RSK, Soner} and the references therein. The convergence to measure-theoretic solutions to the MCF (in the sense of Brakke \cite{Brakke}) was studied by Ilmanen \cite{Ilmanen} and Soner \cite{Soner}, and more recently by Tonegawa \cite{T03}, Mizuno-Tonegawa \cite{MizunoTonegawa}, Pisante-Punzo \cite{PP} and Sato \cite{Sato}. 

In this article, we study such limit flows in the $3$-dimensional round sphere. We show that there are (weak) solution to the MCF connecting Clifford tori to equatorial spheres constructed as the singular limit of solutions to \eqref{PAC} in $S^3$, as $\e \downarrow 0$, and study a family of such limit flows. Some of the motivations for our study are the connections of the Allen-Cahn equation with minimal hypersurfaces from a variational and a dynamical perspective, and recent developments on existence and classification of mean curvature flows, as we briefly discuss next.\medskip

For stationary solutions of \eqref{PAC}, the convergence to minimal hypersurfaces, which are stationary solutions to the MCF, was studied by Modica and Mortola \cite{MM}, Kohn-Sternberg \cite{KS}, Hutchinson-Tonegawa \cite{HT}, Tonegawa-Wickramasekera \cite{TW}, among many others. These results have motivated recent results in the study of minimal hypersurfaces arising as limit interfaces from a variational perspective. 

In \cite{Guaraco}, Guaraco constructed minimal hypersurfaces in closed manifolds using classical variational techniques and the convergence and regularity results of \cite{TW,Wickramasekera} as an alternative to the Almgren-Pitts min-max approach \cite{Pitts}. Very recently, Chodosh and Mantoulidis \cite{ChodoshMantoulidis} obtained curvature estimates for level sets of stable stationary solutions of \eqref{PAC} and obtained strong convergence results of stationary solutions of \eqref{PAC} in 3-dimensional manifolds to minimal surfaces.

The results of Chodosh and Mantoulidis settle a strong version of the \emph{Multiplicity One Conjecture} proposed by F. Cod\'a Marques and A. Neves \cite{MNMult} regarding the multiplicity of min-max minimal hypersurfaces, for generic Riemannian metrics. By applying their results to the stationary solutions of \eqref{PAC} constructed by Guaraco and the second named author in \cite{GG}, Chodosh and Mantoulidis proved the \emph{Morse index conjecture} of Marques and Neves \cite{MNMult}, in dimension 3. We point out that this conjecture was recently solved, in dimensions $3\leq n \leq 7$, by Marques and Neves \cite{MNMorse}, using the solution to the multiplicity one conjecture, by X. Zhou \cite{Zhou}.

Concerning the mean curvature flow, the study of existence, regularity and classification for certain flows in manifolds has been an active topic in geometric analysis. Huisken and Sinestrari \cite{HuiskenSinestrari} showed that closed mean convex ancient flows in the sphere $S^n$ are shrinking spherical caps, provided it satisfies a curvature pinching condition (see \cite{RisaSinestrari} and \cite{LXZ}, and the references therein, for generalizations to higher codimension and to other space forms). In \cite{BIS}, P. Bryan, M.N. Ivaki and J. Scheuer classified ancient convex fully nonlinear flows on $S^n$. We also mention the work of R. Haslhofer and O. Hershkovitz \cite{HH} on the singularities of mean convex MCFs in general ambient manifolds.

In \cite{CM}, K. Choi and C. Mantoulidis proved that ancient smooth MCFs in the round sphere $S^n$ with low area are steady or shrinking spheres. Under more relaxed area bounds, they also proved that such ancient flows in $S^3$ are steady or shrinking equators or Clifford tori, along its unstable directions.\medskip

From a variational and dynamical viewpoint, the equation \eqref{PAC} may be seen as the (negative) gradient flow of the \emph{Allen-Cahn energy functional}
    \[E_{\epsilon}(u)= \int_{M} \left(\frac{\epsilon}{2} |\nabla_{g} u|^2 + \frac{1}{\epsilon} W(u)\right).\]
In this article, we study orbits of this gradient flow in $S^3$ which connect low energy stationary solutions, and describe the mean curvature flows they originate, as $\e \downarrow 0$. We prove:

\begin{thm} \label{thm1}
For sufficiently small $\e>0$, there are eternal solutions $\{u_\e\}$ of \eqref{PAC} such that 
    \[u_\e^{-\infty} = \lim_{t \to -\infty} u_\e(\cdot, t) \quad \text{and} \quad u_\e^{+\infty} = \lim_{t \to +\infty} u_\e(\cdot ,t)\]
are the symmetric critical points of $E_\e$ which accumulate on the Clifford torus and on an equatorial spheres, respectively, as $\e \downarrow 0$. Furthermore, the limit $\{\mu_t\}_{t \in \R}$ of the associated measures
    \[\mu_{\e,t}:=\left(\frac{\e}{2} \, |\nabla u_\e(\cdot, t)|^2 + \frac{W(u_\e(\cdot, t))}{\e} \right)\,d\mu_g\]
is a unit-density Brakke flow on $S^3$ which converges to an equatorial sphere and to the same Clifford torus, as $t \to \pm \infty$, respectively.
\end{thm}

We observe that the flow $\mu_t$ given by the Theorem above is smooth for sufficiently large $|t|$, by Brakke's Local Regularity Theorem \cite{IlmanenRegularization}. In particular, for large negative $t$, such flow belongs to the family of ancient flows constructed by Choi-Mantoulidis in \cite[Theorems 1.6 and 1.7]{CM} which converge exponentially quickly to the Clifford torus as $t \to -\infty$. The $L^1$-integrability condition follows from Proposition 5.3 (see also Remark 4.13) in \cite{CM}, as the Jacobi fields of the Clifford torus are generated by one-parameter families of isometries of $S^3$. \smallskip

We also describe a 2-parameter family of solutions that satisfy the conclusions of Theorem \ref{thm1}, using the symmetries of the gradient flows of the energy and the corresponding mean curvature flows.

\begin{thm} \label{thm2}
Let $\{\e_j\}$ be a sequence of positive parameters such that $\e_j \downarrow 0$. Passing to a subsequence, there exist $2$-parameter families $\{u_j(a)\}_{a \in S^1 \times S^1}$ of solutions of \eqref{PAC}, parametrized by a $2$-torus, which satisfy the conclusions of Theorem \ref{thm1} (for a fixed Clifford torus as their backward limit). These gradient flows, as well as the family of equatorial spheres which are obtained as forward limits, depend continuously and equivariantly with respect to (rotations of) $a \in S^1 \times S^1$.
\end{thm}

\subsection*{Outline of the proof}

Recall that the least area minimal surface on $S^3$ are the totally geodesic equatorial spheres. Furthermore, by the solution of the Willmore conjecture by Marques and Neves \cite{MN}, the second least area minimal surface in $S^3$ is the Clifford torus. Stationary solutions $\{u_\e^{-\infty}\}$ which have this minimal surface as their limit interface can be constructed using a minimization and reflection procedure, see \cite{CGGM}. The rigidity of such solutions was studied by Hiesmayr in \cite{H}.

One can construct ancient solutions $\{u_\e(\cdot, t)\}$ of the (negative) gradient flow of the energy functional which quickly converge backward in time to a critical point using a contraction mapping argument, see \cite{ChodoshMantoulidis}. These solutions describe the unstable manifold of integrable critical points. Then, a topological argument shows that many of these solutions converge, as $t\to +\infty$, to a least energy unstable solutions of the Allen-Cahn equation, which are symmetric critical points that accumulate on equatorial spheres, by \cite{CGGM}.

The convergence of $\{u_\e(\cdot, t)\}$ to an integral Brakke flow on $S^3$ can be derived from \cite{Ilmanen} and \cite{Sato,T03}. This flow is \emph{cyclic mod 2}, in the sense of White \cite{W09}. We then use the area bounds for the limit interfaces obtained from $u_\e^{\pm \infty}$, and the symmetries of the negative eigenfunctions to the Jacobi operator of the Clifford torus (which are inherited by the corresponding solutions of \eqref{PAC}), to describe the forward and backward limit of this flow using the classification of low area stationary varifolds in $S^3$. 

\subsection*{Organization}
In Section \ref{sec:preliminaries}, we state some results concerning the Allen-Cahn equation, minimal surfaces in the $3$-sphere, and the mean curvature flow which will be used in the sequel. In Section \ref{sec:gradient}, we construct low energy eternal solutions of \eqref{PAC} as solutions to the negative gradient flow of the associated energy functional. In Section \ref{sec:main} we study the Brakke flow given as the singular limit of such solutions, and provide a preliminary description of its backward and forward limits. In Section \ref{sec:sym}, we conclude the proof of the main theorems using symmetries of the stability operator of the Clifford torus (and the corresponding solution to the Allen-Cahn equation) to study the asymptotic limits of certain gradient flows of the energy and the corresponding Brakke flow.

\subsection*{Acknowledgements}

We would like to thank Andr\'e Neves for his support, and for many invaluable discussions and suggestions. PG was partially supported by Prof. Neves' Simons Investigator Award.

\subsection*{Notation} \label{notation}
We use the following notation throughout the paper:

\begin{center}
\begin{tabular}{ll}
    $\sigma$ & the energy constant $\int_{-1}^1 \sqrt{W(t)/2} dt$.\\[2pt]
     $\e_2$ & the constant $\lambda_2(S^3)^{-1/2}$ (Section \ref{rigidity}).\\[2pt]
    $W^{1,2}(M)$ & Sobolev space of functions $u \in L^2(M)$ with weak gradient $|\nabla u| \in L^2(M)$.\\[2pt]
    $T_c$ & Clifford torus in $S^3$, \eqref{Clifford}.\\[2pt]
    $u_{\e}(\cdot, t)$,
    $\mathscr{S}(a)(\cdot,t)$ & gradient flow of $E_\e$, solutions of \eqref{PAC} (Sections \ref{sec:gradient} and \ref{sec:sym}).\\[2pt]
    $u_{\e}^{\pm \infty}$ & forward and backward (subsequential) limits of $u_{\e}(\cdot, t)$.\\[2pt]
    $\frac{1}{\sigma} V_{\e,t}$ & associated varifold of $u_{\e}(\cdot, t)$ (Section \ref{ACminimal}).\\[2pt]
    $\frac{1}{\sigma} \mu_{\e,t}$ & corresponding Radon measure of $\frac{1}{\sigma} V_{\e,t}$.\\[2pt]
    $V^{\pm \infty}$ & subsequential limit of the associated varifold of $u_{\e}^{\pm \infty}$.\\[2pt]
    $\Sigma_t$ & limit of $\frac{1}{\sigma} \mu_{\e,t}$, with associated varifold $\frac{1}{\sigma} V_t$.\\[2pt]
    $V_{\pm \infty}$& subsequential limit of $V_t$ as $t \to \pm \infty$\\[2pt]
    $\Theta(V,x)$ & density of a varifold $V$ at a point $x$.
\end{tabular}
\end{center}

\section{Preliminaries} \label{sec:preliminaries}

\subsection{The Allen-Cahn equation, induced varifolds and convergence} \label{ACminimal}

\begin{defi}
A function $W \in C^{\infty}(\R)$ is a \emph{(symmetric) double-well potential} if:
\begin{enumerate}
    \item[(1)] $W$ is nonnegative and vanishes precisely at $\pm1$;
    \item[(2)] $W$ satisfies $W'(0) = 0$, $W''(0) \neq 0$, and $tW'(t) < 0$ for $|t| \in (0, 1)$;
    \item[(3)] $W''(\pm1)>0$;
    \item[(4)] $W(t) = W(-t)$.
\end{enumerate}
\end{defi}

The standard example of a double-well potential is the function $W(t) = \frac{1}{4}(1 - t^2)^2$. Hereafter, we fix such a potential $W$.

\begin{defi}
Let $(M^n,g)$ be a Riemannian manifold. We define the \emph{Allen-Cahn energy} on $\Omega$ by:
\[E_{\epsilon}(u):= \int_{\Omega} \left(\frac{\epsilon}{2} |\nabla_{g} u|^2 + \frac{1}{\epsilon} W(u)\right) d\mu_{g}, \ \ u \in W^{1,2}(M),\]
where $d\mu_g$ is the volume measure with respect to $g$. Note that this quantity is finite provided $W(u) \in L^1(M)$.
\end{defi}

\begin{rmk}
We implicitly assume, in addition to hypotheses (1)-(4) above, that $W$ is bounded. This ensures that the energy functional $E_\e$ is smooth in $W^{1,2}(M)$, and allows us to use existence results and standard estimates for critical points and gradient flows of $E_\e$. We emphasize that this does not affect the arguments explored in this work (in which $M$ is assumed to be compact), as the objects we consider satisfy a priori bounds $|u|<1$, by the maximum principle, so any double-well potential can be modified outside of $[-2,2]$ to meet this requirement.
\end{rmk}

One can check that $u$ is a critical point of $E_{\epsilon}$ on a closed manifold $(M^n,g)$ if and only if $u$ (weakly) solves the \emph{elliptic Allen-Cahn equation}:
\begin{equation} \label{AC}
\epsilon^2 \Delta_g u -  W'(u)=0 \quad \text{on} \ M. \tag{AC}
\end{equation}
For the standard double well potential, the Allen-Cahn equation becomes $\epsilon^2 \Delta_g u = u^3 - u$. 

We write $\sigma = \int_{-1}^1 \sqrt{W(t)/2} dt$. This is the energy of the \emph{heteroclinic solution} $\mathbb{H}_\e(t)$ of \eqref{AC} on $\R$, namely, the unique bounded solution in $\R$ (modulo translation) such that $\mathbb{H}_\e(t) \to \pm 1$ when $t \to \pm \infty$. We refer to \cite[Section 1.3]{ChodoshMantoulidis} for more on this one-dimensional solution.

Recall that the \emph{Morse index} of a solution $u$ of \eqref{AC} (as a critical point of $E_\e$), denoted  $\ind_\e(u)$, is the index of the quadratic form given by the second variation of the energy $E_\e$ at $u$, namely
    \[d^2E_\e[u](\phi,\psi) := \int_M \e\langle \nabla \phi, \nabla \psi \rangle + \frac{1}{\e}W''(u)\phi\psi\ d\mu_g, \]
for $\psi,\phi \in C^{\infty}(M)$. Note that $\ind_\e(u)$ is the number of negative eigenvalues of the linear operator 
    \[\mathcal{L}_{\e,u}(f) = \Delta f - \frac{W''(u)}{\e^2}f,\]
counted with multiplicity. In particular, $\ind_\e(u)$ is finite (note we assumed $M$ to be compact). We also recall that $u$ is said to be a \emph{stable} solution if $\ind_\e(u)=0$.\medskip

In order to describe some convergence results for solutions of \eqref{PAC} and its elliptic counterpart, we will use some notions and notation from Geometric Measure Theory. We refer to \cite{HT} or \cite{TonegawaBook} for a description of such objects and some of their key properties; see also the \nameref{notation} table above.

The classical variational convergence for solutions of \eqref{AC} was studied in the works of Modica and Mortola \cite{M85,MM}, who proved that the Allen–Cahn energy functional $\Gamma$-converges to the \emph{perimeter functional}, a generalization of the $(n-1)$-dimensional volume defined on the space of domains of finite perimeter. In particular, the interfaces of locally minimizing solutions of \eqref{AC} (namely the sets where these functions are bounded away from $\pm 1$) converge, as $\epsilon \downarrow 0$, to local minimizers of the area of the perimeter (and are thus regular away from a singular set of dimension $\leq (n-8)$ ).

A convergence result for families of solutions with uniformly bounded energy and index follows from the combined work of J. Hutchinson, Y. Tonegawa, N, Wickramasekera and M. Guaraco \cite{HT,TW,Guaraco}, which is based on the deep regularity theory developed by Wickramasekera \cite{Wickramasekera}. 

Before recalling this convergence result, we note that given $\e>0$ and a sufficiently regular function $u$ on $M$ (so that almost every level set is a regular hypersurface), we can consider the \emph{associated $(n-1)$-varifolds} $V_{\e,u}$ defined by
    \begin{equation} \label{def:varifold}
        V_{\e,u}(\phi) = \frac{1}{2}\int_{M\cap \{\nabla u \neq 0\}} \phi(x,T_x\{u=u(x)\})\cdot\left(\frac{\e|\nabla u(x)|^2}{2}+\frac{W(u(x))}{\e}\right)\,d\mu_g(x)
    \end{equation}
for any continuous function $\phi$ defined in the Grassmannian manifold $G_{n-1}(M)$, where $V_{\e,u}(\phi)$ denotes the integral of $\phi$ on $G_{n-1}(M)$ with respect to $V_{\e,u}$. We write $\mu_{\e,u}=\|V_{\e,u}\|$ for the associated Radon measure on $M$ (the \emph{weight measure} of $V_{\e,u}$). In the case where $\{u_j\}$ are solutions to \eqref{AC} or \eqref{PAC}, with $\e = \e_j \downarrow 0$, we will write ${V_{\e_j,t}=V_{\e_j,u_j(\cdot,t)}}$ and ${\mu_{\e_j,t} := \mu_{\e_j,u_j(\cdot,t)}}$.

\begin{thm*}[\cite{HT,TW,Guaraco}]
Let $(M^n,g)$ be a closed Riemannian manifold. Let $\{u_j\}$ be a sequence of solutions of \eqref{AC} with $\e=\e_j \downarrow 0$. Suppose that $\sup_j E_{\e_j}(u_j) < \infty$. Then we can find a (not relabeled) subsequence of $u_j$ such that $V_{\e_j}$ converge to a stationary $(n-1)$ varifold $V$ on $M$ such that $\frac{1}{\sigma}V$ is integral. Moreover,
    \[\frac{1}{\sigma}\|V\|(M) = \lim_{j \to \infty} \frac{1}{\sigma} \|V_{\e_j}\|(M) = \lim_{j \to \infty}\frac{1}{2\sigma}E_{\e_j}(u_{j}),\]
and $u_{j}$ converges uniformly to $\pm 1$ in compact subsets of $M \setminus \supp \|V\|$.

Furthermore, if $n\geq 3$ and if $\sup_j \ind_{\e_j}(u_j)<\infty$, then $\supp\|V\|$ is a smooth, embedded, minimal hypersurface in $M$ away from a closed set of Hausdorff dimension $\leq (n-8)$.
\end{thm*}

The minimal surface $\supp\|V\|$ is often called a \emph{limit interface} obtained from $u_j$.

The parabolic counterpart of the measure-theoretic result above was investigated by T. Ilmanen \cite{Ilmanen}, and H.M. Soner \cite{Soner}, among many others. For solutions of the parabolic equation \eqref{PAC}, the weak limit as $\e \downarrow 0$ is a weak solution of the mean curvature flow. The main notion we will employ in this work are \emph{Brakke flows}, measure theoretic solutions of the mean curvature flow defined in terms of varifolds, as introduced by Brakke in \cite{Brakke}. We mention here the  book \cite{TonegawaBook} for a thorough introduction on the Brakke flow.

We state below the main convergence result we will use in the present article, which follows from \cite{Ilmanen} and the work of Tonegawa \cite{T03} (see also \cite{Sato} and \cite{TT}):

\begin{thm*}
Let $(M^n,g)$ be a closed Riemannian manifold. Let $\{u_j\}$ be a sequence of solutions to \eqref{PAC} on $M \times [t_0,\infty)$ with $\e=\e_j \downarrow 0$. Suppose that there exist constants $c_0,E_0>0$ such that
\begin{enumerate}
    \item[(a)] $\sup_{M\times [t_0,\infty)}|u_j| \leq c_0$, for all $j$,
    \item[(b)] $E_{\e_j}(u_j(\cdot,t)) \leq E_0$, for all $t\geq t_0$ and all $j$, and
    \item[(c)] $\int_{M\times(t_0,\infty)}\e_j|\partial_t u_j|^2\,d\mu_g \leq E_0$, for all $j$.
\end{enumerate}
Write $\mu_{\e_j,t} = \mu_{\e,u_j(\cdot, t)}$, for every $t\geq t_0$ and every $j$. Then, passing to a subsequence (not relabeled), there are Radon measures $\{\mu_t\}_{t\geq t_0}$ such that
\begin{enumerate}
    \item[(i)] $\mu_{\e_j,t} \to \mu_t$  as Radon measures on $M$, and
        \[\frac{1}{2\sigma}\lim_{j \to\infty} E_{\e_j}(u_j(\cdot,t)) = \frac{1}{\sigma}\lim_{j \to \infty}\|\mu_{\e_j,t}\|(M) = \frac{1}{\sigma}\|\mu_t\|(M),\]
    for every $t \in [t_0,\infty)$.
    \item[(ii)] For a.e. $t>t_0$, $\mu_t$ is $(n-1)$-rectifiable, and its density is $N(x)\sigma$, for $\mu_t$-a.e. $x \in M$, where $N(x)$ is a nonnegative integer.
    \item[(iii)] $\mu_t$ satisfies the mean curvature flow in the sense of Brakke, namely:
        \[\overline{D}_t \int_M \phi\,d\mu_t \leq \int_M (-\phi)\|H_t\|^2 + \langle\nabla \phi, H_t \rangle \, d\mu_t,\]
    for any $C^2$ function $\phi\geq 0$. Here $\overline{D}_t$ denotes the upper derivative, and $H_t$ is the generalized mean curvature vector of $\mu_t$.
\end{enumerate}
\end{thm*}

\begin{rmk}
The normalization chosen in \eqref{def:varifold} differs by a factor of $\frac{1}{2}$ when compared to \cite{Ilmanen, Soner}, and it agrees with \cite{Sato} (observe the different definition of the normalization constant $\sigma$). Moreover, even though the associated varifold defined in \cite{T03} apparently differs from \eqref{def:varifold}, the convergence of the discrepancy measures $\left|\frac{\e|\nabla u|^2}{2} - \frac{W(u)}{\e}\right|\,d\mu_g$ to zero in the elliptic case \cite{HT} and for a.e. time in the parabolic case \cite{Ilmanen} ensure that the limit varifolds, as $\e \downarrow 0$ agree.
\end{rmk}

\begin{rmk} We mention here the recent results of \cite{NW2,NW1}, by H.T. Nguyen and S. Wang, regarding the strong convergence of solutions to \eqref{PAC} in the Euclidean space. This parallels the results of \cite{ChodoshMantoulidis} for the parabolic setting, under entropy bounds or multiplicity one conditions. Even though we do not use these results in the present article, we point out that one could employ them to obtain uniform curvature bounds (for sufficiently negative time) for the transition layers of the solutions to \eqref{PAC} studied here.
\end{rmk}

\subsection{Minimal surfaces in \texorpdfstring{$S^3$}{S3} and rigidity of solutions} \label{rigidity}

The second variation formula for the area functional (see e.g. \cite{S}) shows that any embedded minimal surface in $S^3$ cannot be stable (that is, its Morse index is $\geq 1$). It is well known that totally geodesic equators in $S^3$ are the closed minimal surfaces of least area. Simons characterized the equator as the only minimal surface in $S^3$ with index one. Furthermore, Urbano \cite{U} proved that if $\Sigma \subset S^3$ is a minimal surface which is not an equator, then its index is at least $5$, and the \emph{Clifford torus} is the only closed embedded minimal surface whose index is precisely $5$. We recall that this is, up to isometry, the minimal surface 
\begin{equation} \label{Clifford}
T_c = \left\{(x,y,z,w) \in \R^4: x^2 + y^2 = z^2 + w^2 = \frac{1}{2}\right\} \subset S^3.
\end{equation}

By the solution of the Willmore conjecture, by Marques-Neves \cite{MN}, this is the embedded, non-totally geodesic, minimal surface of least area $2 \pi^2$ in $S^3$. Moreover, it is the unique minimally embedded torus in $S^3$ up to isometries, by the Hsiang–Lawson's conjecture, recently solved by S. Brendle \cite{Brendle}. We also mention here that if $T$ is a $2$-dimensional stationary integral varifold in $S^3$ with $\|T\|(S^3) \leq 2 \pi^2$ such that its associated $\mathbb{Z}_2$ chain $[T]$ has $\partial[T] = 0$ (see \cite{W09}), then $T$ is either a multiplicity one equator or a Clifford torus, see e.g. \cite[Lemma 5.8]{CM}. Hereafter, we will refer to any such minimal torus as a Clifford torus, while reserving the notation $T_c$ for the specific torus described above. \smallskip

We now describe the counterparts of some of these results in the context of the Allen-Cahn equation. First, we recall that $u = \pm 1$ are the unique global minimizers for $E_{\epsilon}$. Low energy solutions often display variational characterizations and inherit many geometrical properties from the domain. The symmetry properties of \emph{least energy unstable} critical points of $E_\e$ -- also referred as \emph{ground state solutions} -- in a sphere were studied in \cite{CGGM}. Such solutions are radially symmetric with respect to some point, and they vanish precisely on an equatorial sphere.

Solutions of the Allen-Cahn equation whose energy density accumulate on $T_c$ can be constructed using gluing techniques (see \cite{CG}), or by minimization and reflection. Concretely, for each $\e \in (0,\e_2)$, where 
    \begin{equation} \label{e_2}
        \e_2 = 1/\lambda_2(S^3)^{1/2},
    \end{equation}
there is a unique positive function which minimizes $E_\e$ on one of the two isometric domains in $S^3$ bounded by $T_c$ among functions that vanish on the boundary. Using the reflection that maps one of the domains onto the other, one extends this minimizer to a solution of \eqref{AC} in $S^3$ which vanishes precisely on $T_c$. We refer to \cite{CGGM} or \cite{H} for the detailed construction. Note that the uniqueness of this minimizers implies that this solution inherits the symmetries of the Clifford torus.

As noted in \cite{CG}, the index estimates from \cite{CM} and \cite{GIndex} imply that the solutions produced by either of these methods also have Morse index $5$, for sufficiently small $\e>0$. Very recently, F. Hiesmayr \cite{H} characterized solutions whose nodal sets are equators or Clifford torus as the unique nonradial solutions of Morse index $\leq 5$ in $S^3$ with bounded energy, namely:

\begin{thm*}[\cite{H}]
Given any $C>1$, there exists $\e_3(C) \in (0,\e_2)$ with the following property. For any $\e \in (0,\e_3)$, any solution of \eqref{AC} with Morse index $\leq 5$ and energy $\leq C$ is a ground state solution or a symmetric solution with nodal set on some Clifford torus.
\end{thm*}

\subsection{Ancient mean curvature flows in \texorpdfstring{$S^3$}{S3}}

In \cite{CM}, among many other results, K. Choi and C. Mantoulidis classified smooth ancient MCFs in $S^3$ with area below $2\pi^2$ plus a small $\delta>0$. Essentially, these flows are either steady or shrinking equators along spheres of latitude, or steady or decreasing tori along one of its $5$ linearly independent directions (see \cite[Corollary 1.5]{CM}). 

In the same article, Choi and Mantoulidis constructed ancient solutions to certain quasilinear gradient flows that converge backward in time --- that is, as $t \downarrow -\infty$ --- to a  critical point $u$ for the associated energy with finite Morse index $p$. These solutions locally describe the unstable manifold of this critical point, hence they are parametrized by a $p$-dimensional disk \cite[Theorem 3.3]{CM}, and they are the unique solutions that converge backward to $u$ that satisfy an integrability condition related to the \emph{\L ojasiewicz-Simon inequality}, as proved in \cite[Theorem 4.1 and Proposition 4.12]{CM}. In particular, this assumption holds true for analytic functionals, and also whenever the energy functional is Morse-Bott at the corresponding energy level near the critical point $u$, see \cite{FM}. 

We refer to Sections 3 and 4 in \cite{CM} for the complete statements, and to Sections \ref{sec:gradient} and \ref{sec:sym} for the main consequences to the (negative) gradient flow of the Allen-Cahn energy.

\section{Gradient flows of the energy functional} \label{sec:gradient}

As noted above, the Clifford torus $T_c \subset S^3$ can be obtained as the limit of a sequence of solutions to the Allen-Cahn equation \eqref{AC}, extracted from a family $\{u^{-\infty}_{\epsilon}\}_{\epsilon \in (0, \epsilon_2)}$ of solutions whose nodal sets are precisely $T_c$, and such that $E_{\epsilon} (u^{-\infty}_{\epsilon}) \to 2\sigma \cdot \mathrm{Area}(T_c)$ as $\epsilon \downarrow 0$.

We want to show that we can connect these solutions to a \emph{nonconstant} solution of the Allen-Cahn equation with smaller energy using the negative gradient flow of the energy functional. The key ingredient is the following result:

\begin{prop}\label{prop_parabolicflow} Let $(M^n,g)$ be a compact Riemannian manifold with $\Ric_g \geq 0$, and let $u_\e^{-\infty}$ be a nonconstant solution of the Allen-Cahn equation on $M$ with Morse index $\mathrm{ind}_\e(u_\e^{-\infty})\geq 2$. There exist (infinitely many) eternal solutions $u:M \times \R \to \R$ of the parabolic equation \eqref{PAC} on $(M,g)$ such that $E_{\e}(u_\e(\cdot,t))$ is strictly decreasing,
	\[\|u_\e(\ \cdot \ , t) - u_\e^{-\infty}\|_{W^{1,2}(M)} \to 0, \quad \text{as} \quad t \to -\infty,\]
and, for any sequence $t_k \uparrow +\infty$, the functions $u_\e(\cdot,t_k)$ do not converge to the constant critical points $\pm 1$ of $E_\e$.
\end{prop}

Observe that a time-dependent function $u$ is a (weak) solution to \eqref{PAC} if, and only if, is a (weak) solution to the negative $L^2$-gradient flow of $-\frac{1}{\e}E_\e$,  that is
	\begin{equation} \label{parabolic_Aflow}
		\partial_t u = -\frac{1}{\e}\nabla E_\e(u).
	\end{equation}
A possible parametrization of solutions of such gradient flows near critical points of the associated energy functional is developed in \cite[Theorem 3.3]{CM}. In order to describe their results, we will introduce some notation. 

Denote by $\mathcal{L}_\e$ the linearization of $-\frac{1}{\e}\nabla E_\e$:
	\[\mathcal{L}_\e(f) = -{\textstyle \frac{1}{\e}\frac{d}{ds}\big|}_{s=0}\nabla E_\e(sf).\]
One shows that this is related to the linearized Allen-Cahn operator at $u_\e^{-\infty}$ by:
	\[\mathcal{L}_\e(f) = \Delta f - \frac{W''(u_\e^{-\infty})}{\epsilon^2}f.\]
Let $\lambda^\e_1< \lambda^\e_2 \leq \ldots \leq \lambda^\e_I$ be the negative eigenvalues of $\mathcal{L}_\e$, where $I$ is the Morse index of $u^{-\infty}_\e$, and let $\varphi_1,\ldots,\varphi_I$ be corresponding $L^2$-orthonormal eigenfunctions with $\varphi_1>0$. 

In \cite{CM}, K. Choi and C. Mantoulidis proved that there exists $\eta=\eta(\e,M)>0$ and a $C^{2,\alpha}$-continuous family ${\{\mathscr{S}(a): M \times (-\infty,0]\to \R\}}$, for $a \in B_\eta(0) \subset \R^I$, of ancient solution to the parabolic equation \eqref{parabolic_Aflow} with controlled exponential decay (as $t \downarrow -\infty$). The solution $\mathscr{S}(a)$ is the unique $C^{1,\theta}$ solution of \eqref{parabolic_Aflow} with finite $L^1(M \times (-\infty,0])$ norm, modulo translation in time, that converges to $u_{\e}^{-\infty}$ backward in time, in the $C^{2,\theta}$ norm, and has $u_\e^{-\infty}+\sum_{j=1}^I a_j \varphi_j$ as its projection in the space generated by $\{\varphi_j\}_{j=1}^I$ at time $t=0$. Moreover, it satisfies
		\begin{equation} \label{flow_estimate}
			\left\|\mathscr{S}(a)(\cdot,0) - \left(u_\e^{-\infty}+\sum_{j=1}^I a_j \varphi_j\right)\right\|_{C^{2,\alpha}(M)} \leq C|a|^2,
		\end{equation}
for some $C>0$ (depending on $\e$) -- see Theorems 3.3 and 4.1 in \cite{CM} for a more precise and complete statement. By picking a sufficiently small $\eta>0$, we may assume that $\sup_{M}|\mathscr{S}(a)(\cdot,0)|<1$ for all $a \in B_\eta(0)$, and that 
	\[\mathscr{S}(r,0\ldots,0)(\cdot, 0)>u_\e^{-\infty} > \mathscr{S}(-r,0,\ldots,0)(\cdot,0) \ \text{for all} \ r \in (0,\eta).\]
This is possible because the first eigenfunction $\varphi_1$ is positive and by
	\[\left\|\frac{\mathscr{S}(a)(\cdot,0) - u_\e^{-\infty}}{r} - (\pm \varphi_1) \right\|_{C^{2,\alpha}(M)} \leq C r\]
for $a=(\pm r,0,\ldots,0)$, which follows from the estimate \eqref{flow_estimate}. We are now in position to prove Proposition \ref{prop_parabolicflow}.

\begin{proof}[Proof of Proposition \ref{prop_parabolicflow}]
Since the constant functions $\pm 1$ are isolated global minimizers of $E_{\epsilon}$, there exist disjoint neighborhoods $B_\pm$ of $\pm 1$ in $W^{1,2}(M)$ and $d_\e >0$ such that $\pm 1$ are the only solutions of \eqref{AC} in $B_{\pm}$, and
	\[E_\e(u) < E_\e(\pm 1) + d_\e = d_\e \quad \text{if, and only if}, \quad u \in B_{\pm}.\]

Let $\mathcal{S} = \{u \in C^2(M) \mid |u|\leq 1\}$. By Lemma 2.3 in \cite{GG} (and the continuous dependence of initial data, see e.g. Cazenave-Haraux \cite{CH}), there is a continuous map 
	\[\Phi:\mathcal{S} \times [0,\infty) \to W^{1,2}(M)\]
such that $\Phi(u,\cdot):M \times [0,\infty) \to \R$ is a solution of \eqref{PAC} defined for all $t \geq 0$ with $\Phi(u,0) = u$, and such that $\Phi(u,t) \in \mathcal{S}$, for all such $t$. Since $E_\e(\Phi(u,t))$ is decreasing with respect to $t$, for any $u \in \mathcal{S}$ and for any $T>0$,
	\[\text{if}\quad  \Phi(u,T) \in B_{\pm}, \quad \text{then} \quad \Phi(u,t) \in B_{\pm}, \ \text{for all} \ t\geq T.\]

We claim that the sets
	\begin{align*}
		U_{\pm} = \{u \in \mathcal{S} \mid \|\Phi(t,u)- (\pm 1)\|_{W^{1,2}(M)} \to 0, \ \text{as} \ t\to +\infty\}
	\end{align*}
are open. In fact, if $u \in U_{\pm}$, then there exists $T>0$ such that $\Phi(u,t) \in B_{\pm}$ for all $t \geq T$. By the continuity of $\Phi(\cdot,T)$, there exists $\delta>0$ such that $\Phi(w,T) \in B_{\pm}$ for all $w \in B_\delta(u)\cap \mathcal{S}$. This implies $\Phi(w,t) \in B_{\pm}$ for all $t \geq T$, so $\Phi(w,t)$ must converge to $\pm 1$ as $t \to +\infty$.

With the notation introduced above, let $r \in (0,\eta)$ and consider  $K=\{\mathscr{S}(a)(\cdot,0)\}_{a \in \partial B_r(0)}$. Write $w^\pm = \mathscr{S}(\pm r,0\ldots, 0)(\cdot, 0)$. By our choice of $\eta$, we have $w^+ > u_\e^{-\infty} > w^-$. As noted in \cite[Lemma 2.3]{GG}, there exist sequences $(t^\pm_k)$ such that $t_k^\pm \to \pm \infty$, and solutions $u^\pm$ of the Allen-Cahn equation such that $\Phi(w^\pm,t_k^\pm) \to u^\pm$. By the maximum principle for parabolic equations, these solutions satisfy $u^+ \geq u_\e^{-\infty} \geq u^-$, as well as $E_\e(w^\pm) < E_\e(u_\e^{-\infty})$. Since $u_\e^{-\infty}$ is nonconstant, by Proposition 6 in \cite{H}, we see that $u^\pm \equiv \pm  1$. Since $\pm 1$ are nondegenerate solutions, we get $\Phi(w^\pm,t) \to \pm 1$ as $t \to \pm \infty$ (see Remark 4.13 in \cite{CM}) and $w^\pm \in U_\pm$. Therefore, the sets $U_+ \cap K$ and $U_- \cap K$ are nonempty. 

By connectedness, it follows that along any path in $K$ that connects $w^{\pm}$, there exist $w=\mathscr{S}(a)(\cdot, 0) \in K$ also in this path, for some $a \in \partial B_r(0)$, such that $w \notin U_\pm$. This means that $\Phi(w,t)$ does not converge to $\pm 1$, as $t \to +\infty$. The desired solution of the parabolic equation \eqref{PAC} is then given by $u_\e(\cdot,t) = \mathscr{S}(a)(\cdot,t)$, for $t \leq 0$, and $u_\e(\cdot,t) = \Phi(w,t)$, for $t \geq 0$. Since $u_\e^{-\infty}$ has Morse index $\geq 2$, the set $K$ is a continuous injective image of a sphere of dimension $(\ind_\e(u_\e^{-\infty})-1)\geq 1$, hence there are infinitely many such paths, and this conclude the proof.
\end{proof}

In the next section, we will focus on the case where $u_\e^{-\infty}$ is the solution of \eqref{AC} in $S^3$ which vanishes precisely at the Clifford torus $T_c$, has this minimal surface as its limit interface. We hope to prove that $u_\e(\cdot,t)$ converges to a solution of the elliptic Allen-Cahn equation which vanishes precisely on an equatorial sphere. We will need to work around the fact that any (subsequential) limit of $u_\e(\cdot,t)$ as $t \to +\infty$ may have larger Morse index, so the regularity result from \cite{TW,Guaraco} does not readily apply. Without imposing any further conditions on the gradient flow of the functional, this phenomenon may happen even in the finite-dimensional setting, as we illustrate with an example below.

\subsection{Index change related to the gradient flow}

Let $F$ be a $C^1$ function on a manifold $M$. If $f: \R \to M$ is a complete solution to the negative gradient flow of $F$ on $M$ which joins a critical point $x$ to another critical point $y$, we know that $F(x) > F(y)$, but we may not have a relation between the Morse indexes of $x$ and $y$. In fact, even in finite dimensions, there are examples in which the Morse index increases along the flow, that is, such that $\mathrm{index}(x) < \mathrm{index}(y)$.

\begin{ex}
Let $M = S^1 \times S^2$ equipped with coordinates $(x,y,z,w,u)$ where $x^2 + y^2 = z^2 + w^2 + u^2 = 1$. Let $F(x,y,z,w,u) = y(u+2)$, then $F$ has $4$ critical points on $M$:
\begin{align*}
&(0,-1,0,0,\ \ 1),\ \text{with Morse index}\ 0, \\
&(0,\ \ 1,0,0,-1),\ \text{with Morse index}\ 1, \\
&(0,-1,0,0,-1),\ \text{with Morse index}\ 2, \\
&(0,\ \  1,0,0,\ \ 1), \ \text{with Morse index}\ 3.
\end{align*}

Then $f(t) = (0,\sin t, 0,0,-1), t \in [-\frac{\pi}{2}, \frac{\pi}{2}]$ is a gradient flow of $F$ from the index $1$ critical point to the index $2$ critical point.
\end{ex}

\section{Limit flows and interfaces} \label{sec:main}

Recall that if $u_{\e}$ denotes a solution to \eqref{PAC} on $S^3 \times I$, $I \subset \R$, we can define associated $2$-varifolds $V_{\epsilon,t}=V_{\e,u_\e(\cdot,t)}$ given by
    \[V_{\e,t}(\phi) = \frac{1}{2}\int_{S^3 \cap \{\nabla u_\e(\cdot,t) \neq 0\}} \phi(x,T_x\{u_{\e}=u_{\e}(x,t)\})\cdot \left( \frac{\e|\nabla u_{\e}(x,t)|^2}{2} + \frac{W(u_\e(x,t))}{\e} \right)\,d\mu_g(x)\]
for any continuous function $\phi$ on $G_{2}(S^3)$, and Radon measures $\mu_{\epsilon,t}=\mu_{\e,u_\e(\cdot,t)}$ on $S^3$ given by the weight measure $\mu_{\e,t}=\|V_{\e,t}\|$ of $V_{\e,t}$.

Throughout this section, for $\e \in (0,\e_2)$, we consider any solution $u_{\e}\colon S^3 \times \R \to \R$ which satisfies the conclusions of Proposition \ref{prop_parabolicflow}, where $u_\e^{-\infty}$ is a solution of \eqref{AC} whose nodal set is the Clifford torus, $\{u_\e^{-\infty}=0\} \cap S^3 = T_c$ (recall that this solution is unique, modulo sign).

\subsection{Asymptotic convergence of the gradient flow}
In what follows, we denote by $u_\e^{+\infty}$ an arbitrary subsequential limit of $u_\e(\cdot,t_k)$, for a sequence $t_k \to +\infty$. Note that this can be extracted from any such sequence, as consequence of the uniform energy bounds along the flow and the compactness properties of the Sobolev space $W^{1.2}$ \cite[Lemma 2.3]{GG}. Our main goal in this subsection is to prove that, for sufficiently small $\e>0$, this limit is a ground state and it is unique, using information about its limit interface. The main ingredient is the following rigidity result. \smallskip

\begin{lem} \label{equatorial}
There exists $\e_4 \in (0,\e_2)$ with the following property. For any $\e \in (0,\e_4)$, the solution $u_\e^{-\infty}$ has Morse index 5, and the only nonconstant solutions of \eqref{AC} with energy $<E_\e(u_\e^{-\infty})$ are ground states (in the sense defined in \cite{CGGM}).
\end{lem}

\begin{proof}
The existence of $\e_4$ satisfying the first property follows from the index bounds of \cite{GIndex,HIndex} and \cite{ChodoshMantoulidis} under the multiplicity one hypothesis (see Section 4.3 in \cite{CG}), and the fact that the Clifford torus has Morse index $5$.

Suppose that there does not exist such $\e_4>0$. Then there exist a sequence $\e_j>0$ such that $\e_j \downarrow 0$, and a sequence $u_j$ of nonconstant solutions to \eqref{AC} with $\e=\e_j$ which are not ground states and have energy $<E_{\e_j}(u_{\e_j}^{-\infty})$ . Since $u_j$ is unstable, if $a_j$ is the energy of a ground state, then $E_{\e_j}(u_j) > a_j$.

Since these solutions have uniformly bounded energy, by passing to a subsequence, we may assume that the varifolds $\frac{1}{\sigma}V_{\e_j,u_j}$ converge to a stationary integral varifold $\frac{1}{\sigma}V$ in $S^3$ with area in $[4\pi,2\pi^2]$. We claim that $\frac{1}{\sigma}V$ has density $\mathcal{H}^2$-a.e. equal to $1$ on its support. In fact, $\Theta(\frac{1}{\sigma}V,x) \in \mathbb{Z}_{+}$ for $\mathcal{H}^2$-almost every such $x$. From $\frac{1}{\sigma}\|V\|(S^3)\leq 2\pi^2$ and the density estimate in \cite[Lemma A.2]{MN}, we see that the density of $\frac{1}{\sigma}V$ is everywhere strictly less than $2$, proving the claim.

By the remarks about energy loss in \cite{HT}, we see that $\frac{1}{\sigma}V$ is the boundary of a region. More precisely, the varifold $\frac{1}{\sigma}V$ agrees with the multiplicity one varifold induced by the reduced boundary of $\{u=1\}$, where $u$ is the function of bounded variation on $S^3$ given by the a.e. limit of $u_j$. Consequently, the boundary of the $\mathbb{Z}_2$ chain associated to $\frac{1}{\sigma}V$ (in the sense of White \cite{W09}) vanishes. By the aforementioned result from Choi-Mantoulidis \cite{CM}, it follows that $\frac{1}{\sigma}V$ is either a multiplicity one equatorial sphere or a Clifford torus.

To conclude the proof, we use the rigidity result of Hiesmayr and index bounds for the limit interface $\supp\|V\|$ to reach a contradiction. These index bounds imply that $u_j$ has Morse index $1$ or $5$ for sufficiently large $j$. Then $u_j$ is either a ground state solution or equal to $u_{\e_j}^{-\infty}$, up to isometries. This contradicts the area bounds $a_j < E_{\e_j}(u_j)< E_{\e_j}(u_{\e_j}^{-\infty})$.
\end{proof}

We apply this result to the subsequential limit $u_\e^{+\infty}$, assuming $\e \in (0,\e_4)$. Since the energy of $E_\e(u_\e(\cdot,t))$ is strictly decreasing, we have
    \[\sup_\e E_\e(u_\e^{+\infty}) < \sup_\e E_\e(u_\e^{-\infty}) <\infty,\]
so $u_\e^{+\infty}$ is a ground state solution, by Lemma \ref{equatorial}. In particular, if $V_{\epsilon}^{+\infty}$ is the varifold associated to $u^{+\infty}_{\epsilon}$, then $V^{+\infty}_{\epsilon} \to V^{+\infty}$ subsequentially, in the varifold sense, where $\frac{1}{\sigma}V^{+\infty}$ is a multiplicity one equatorial sphere.

We are ready to prove the full convergence of $u_\e(\cdot,t)$ to $u_\e^{+\infty}$ using the previous Lemma and the Łojasiewicz-Simon inequality. This is stated, more precisely in the following proposition, which summarizes the results obtained in this subsection.

\begin{prop} \label{prop2}
Let $\e_4>0$ be given by Lemma \ref{equatorial}, and consider the critical points $u_\e^{-\infty}$ of $E_\e$ described above. Given $\e \in (0,\e_4)$, suppose that $u_{\epsilon}:S^3\times \R \to \R$ is a sequence of eternal solutions to \eqref{PAC} on $S^3$ such that
\begin{enumerate}
    \item[(i)] $E_{\e}(u_\e(\cdot,t))$ is strictly decreasing;
    \item[(ii)] $\|u_\e(\ \cdot \ , t) - u_\e^{-\infty}\|_{W^{1,2}(M)} \to 0$ as $t \to -\infty$;
    \item[(iii)] for any sequence $t_k \to +\infty$, the functions $u_\e(\cdot,t_k)$ do not converge to the constant critical points $\pm 1$ of $E_\e$.
\end{enumerate}
Then, there exists a (nonconstant) least energy unstable solution $u^{+\infty}_{\e}$ of the Allen-Cahn equation \eqref{AC} such that 
	\[\|u_{\e}(\ \cdot \ , t) - u_{\e}^{\pm\infty}\|_{W^{1,2}(S^3)} \to 0, \quad \text{as} \quad t \to \pm\infty.\]
In particular, any limit interface obtained from the limits $u_\e^{+\infty}$ is a multiplicity one equatorial sphere, and it holds $E_{\e}(u_{\e}^{+\infty}) \to (2\sigma)\cdot 4\pi$.
\end{prop}

\begin{proof} It remains to prove the full convergence of $u_\e(\cdot,t)$, as $t \to +\infty$. By \cite{CGGM}, least energy unstable critical points of the energy functional $E_\e$ are unique up to ambient isometries. In particular, $E_\e$ is a \emph{Morse-Bott} functional at this critical level. Thus, the convergence of $u_\e(\cdot, t)$ to $u_\e^{+\infty}$ in $W^{1,2}$ is a consequence of the Łojasiewicz-Simon gradient inequality for such functionals, see e.g. \cite{FM}.
\end{proof}

\subsection{The limit flow} Now we analyze the limit of the gradient flow given by Proposition \ref{prop2} as $\e\downarrow 0$. For simplicity, we will omit the index $j$ in sequences $\e_j\downarrow 0$ and in the corresponding objects (the energy functional, functions, and varifolds). First, we aim to prove that the gradient flow satisfies the necessary conditions to take the limit as $\e\downarrow 0$ and obtain a codimension one Brakke flow in the sphere $S^3$.

We showed that $E_{\epsilon} (u_{\epsilon}^{-\infty}) \to 2\sigma (2\pi^2)$, and $E_{\epsilon}(u_\epsilon^{+\infty}) \to 2\sigma (4\pi)$. Thus given a small $\delta > 0$, for sufficiently small $\epsilon > 0$ (depending on $\delta$), we have 
    \[2\sigma (2\pi^2) - \delta \leq E_{\epsilon} (u_{\epsilon}^{-\infty}) \leq 2\sigma (2\pi^2) + \delta \quad \text{and} \quad E_{\epsilon} (u_{\epsilon}^{+\infty}) \leq 2\sigma (4\pi) + \delta.\]
Recall that the energy $E_{\epsilon}(u_{\epsilon}(\cdot,t))$ is a continuous strictly decreasing function of $t$. By picking a sufficiently small $\delta>0$ and by noting that this solution joins $u_{\epsilon}^{-\infty}$ to $u_{\epsilon}^{+\infty}$, we see that there exists $t(\epsilon) \in \R$ such that $E_{\epsilon}\left(\, u_{\epsilon}(\cdot,t(\epsilon))\,\right) = 2\sigma (5\pi)$ (as $4\pi < 5\pi < 2\pi^2$). By translating the gradient flow $u_\epsilon(\cdot,t)$ to $u_{\epsilon}(\ \cdot \ , t+t(\epsilon))$, we can assume that $E_{\epsilon}\left(\,u_{\epsilon}(\cdot,0)\,\right) = 2\sigma (5\pi)$ for all small $\epsilon$.

\begin{rmk}
In this section, we will abuse notation and also denote by $u_{\e_j}$ (or simply $u_\e$) the time-translated solution. When considering a \emph{family} of solutions to \eqref{PAC} -- in particular the family $\mathscr{S}$ constructed in \cite{CM} and described in Section \ref{sec:gradient} -- it will be useful to keep track of this translation to contrast it with the initial condition $u_\e(\cdot,0)$.
\end{rmk}

\begin{lem} \label{bounds}
There exists $\e_5 \in (0,\e_4)$ with the following property. Let $u_{\epsilon}\colon M \times \mathbb{R} \to \mathbb{R}$ be a solution of \eqref{PAC} that satisfies the hypotheses of Proposition \ref{prop2}, for $\e \in (0,\e_5)$. Then $u_{\e} \in C^{3}(S^3\times \R)$ and $\sup_{S^3 \times \R} |u_{\e}|< 1$. Moreover, 
\begin{enumerate}
    \item For any $t \in \R$,  it holds
        \[E_{\e}\left(\, u_{\e}(\ \cdot \ ,t)\,\right) \leq 2\sigma(2\pi^2) +1\]
    \item It holds
        \[\int_{S^3\times \R}\e|\partial_t u_{\e}|^2 \leq 2\sigma(2\pi^2)+1.\]
\end{enumerate}
\end{lem}

\begin{proof}

Once again, we will omit the index $j$. Since $u_{\epsilon}^{-\infty}$ and $u_{\epsilon}^{+\infty}$ are solutions of \eqref{AC}, we have that $|u_{\epsilon}^{-\infty}| \leq 1$ and $|u_{\epsilon}^{+\infty}| \leq 1$. Since $\{u_{\epsilon}\}$ solves the parabolic Allen-Cahn equation \eqref{PAC}, by a maximum principle argument (i.e. Lemma 2.3 (2) in Gaspar-Guaraco \cite{GG}), we know that $|u_{\epsilon}| < 1$.

As noted above, for sufficiently small $\e>0$ (depending only on $E_{\e}(u_\e^{+\infty})$), the energies of $u_{\epsilon}(\cdot,t)$ are bounded above by $2\sigma (2\pi^2) + 1$, which we will denote by $E_0$. By the monotonicity of the energy, this proves (1). Furthermore, $E_\e(u_{\e})$ is differentiable with respect to $t$ and

\begin{align*}
\frac{d}{dt} E_{\epsilon}(u_{\epsilon}) &=\int_{S^3} \epsilon \left\langle\nabla u_{\epsilon}, \nabla (\partial_t u_{\epsilon})\right\rangle + \frac{1}{\epsilon} W'(u_{\epsilon}) \partial_t u_{\epsilon} \\ 
&= \int_{S^3} -\epsilon (\partial_t u_{\epsilon})(\Delta u_{\epsilon} - \frac{1}{\epsilon^2} W'(u_{\epsilon})) \\
&= -\int_{S^3} \epsilon |\partial_t u_{\epsilon}|^2.
\end{align*}

Thus, for all $T_1<T_2$, 
\begin{align*}
\int_{S^3 \times (T_1,T_2)} \epsilon |\partial_t u_{\epsilon}|^2 = -\int_{T_1}^{T_2} \frac{d}{dt} E_{\epsilon}(u_{\epsilon}) dt = E_\e(u_\e(\cdot,T_1)) - E_\e(u_\e(\cdot,T_2)) \leq E_\e(u_\e^{-\infty})< E_0.
\end{align*}
This proves (2).
\end{proof}

By the convergence result for solutions of \eqref{PAC} of Ilmanen \cite{Ilmanen} and Tonegawa \cite{T03} (see also Sato \cite{Sato}), after passing to a subsequence (not relabeled) with $\e\downarrow 0$, the varifolds $V_{\e,t}$ associated to $u_{\epsilon}(\cdot,t)$ converge, for every $t \in \R$, to a $2$-varifold $V_t$, and the underlying Radon measures $\frac{1}{\sigma}\mu_{\e,t}=\frac{1}{\sigma}\|V_{\e,t}\|$ converge to a Radon measure       \begin{equation}\Sigma_t:=\frac{1}{\sigma}\mu_t=\frac{1}{\sigma}\|V_t\| \end{equation}
which satisfies the mean curvature flow equation in the sense of Brakke. Moreover,
     \[\frac{1}{2\sigma} E_{\epsilon}\left(\,u_{\epsilon}(\ \cdot \ ,t)\,\right) \to \|\Sigma_t\| (S^3), \quad \text{as} \quad \epsilon \downarrow 0,\]
and, for almost every $t \in \R$, the varifold $\frac{1}{\sigma}V_t$ is an integral varifold.\smallskip
   
More precisely, we apply the convergence result to $u_\e$ on $S^3 \times [-m,+\infty)$ for each $m \in \mathbb{N}$ to obtain the (subsequential) convergence of $\frac{1}{\sigma}V_{\e,t}$ for all $t \geq -m$. By picking a diagonal subsequence, we get the convergence described above.\medskip

Since $\frac{1}{2\sigma} E_{\epsilon}\left(\,u_{\epsilon}(\ \cdot \ ,0)\,\right) = 5\pi$ for all small $\epsilon$, we see that 
    \[\|\Sigma_{-t}\|(S^3) = \lim_{\e \downarrow 0} \frac{1}{2\sigma}E_{\e}(u_\e(\cdot,-t)) \geq \lim_{\e \downarrow 0} \frac{1}{2\sigma}E_{\e}(u_\e(\cdot,0)) = 5\pi\]
and, similarly, $\|\Sigma_t\|(S^3) \leq 5\pi$, for every $t \geq 0$. Note also that $\|\Sigma_t\|(S^3) \leq 2\pi^2$ for all $t \in \R$. In fact, since $\frac{1}{2\sigma}E_{\epsilon} (u^{-\infty}_{\epsilon}) \to 2\pi^2$ as $\epsilon \downarrow 0$, for each $\delta' > 0$, there exists a $\epsilon' > 0$ small, such that 
    \[\frac{1}{2\sigma}E_{\epsilon} (u^{-\infty}_{\epsilon}) < 2\pi^2 + \delta', \quad \text{for all} \  \epsilon \in (0,\epsilon').\]
By noting that $E_{\epsilon}\left(\, u_{\epsilon}(\ \cdot 
\ ,t)\,\right)\leq E_\e(u_\e^{-\infty})$ for all $t$, we obtain $\|\Sigma_t\|(S^3) \leq 2\pi^2 + {\delta'}$, for every $t$. Since $\delta'$ is arbitrary, this implies that $\|\Sigma_t\|(S^3)\leq 2 \pi^2$. \smallskip

By abuse of notation, we will identify $\Sigma_t$ with its support, which is, for almost every $t \in \R$, a $2$-dimensional rectifiable set. We want to show that $\Sigma_t$ (and the associated varifolds $\frac{1}{\sigma}V_t$) converge to a multiplicity one Clifford torus, as $t \to -\infty$  along subsequences, and to a multiplicity one equatorial sphere, as $t \to \infty$, also along subsequences. Note that in general we don't know whether or not $\Sigma_t$ converges to the initial torus $T_c=\{u_\e^{-\infty}=0\}\cap S^3$ as $t \to -\infty$, and $\Sigma_t$ converges to $\frac{1}{\sigma}V^{+\infty}$ (the $\e$-limit of $V_{\e,u_\e^{+\infty}}$) as $t \to +\infty$. This question will be addressed in Section \ref{sec:sym} using symmetries of $S^3$. \medskip

We will need further information about the parity of the multiplicity of the limit varifold $V_t$, as described in the following lemma. It intuitively says that the interfaces fold an even number of times near a point in $\supp\|V_t\|$ if, and only if, $u_\e(\cdot, t)$ converges to the same value on the two sides of this surface. This was proved by Hutchinson-Tonegawa \cite{HT} in the elliptic case; see also Takasao-Tonegawa \cite{TT} in the parabolic case, for an equation with a transport term in Euclidean domains or in a torus.

\begin{lem} \label{odd-even-density}
For almost every $t\in \R$, the density of the varifold $\Sigma_t$ satisfies 
\[\Theta(\Sigma_t,x)=\left\{
\begin{aligned}
\rm{odd}&  \ \ \  \mathcal{H}^{2}\text{-a.e.} \  x \in M_t, \\
\rm{even}&  \ \ \  \mathcal{H}^{2}\text{-a.e.}\ x \in \mathrm{supp}\|\Sigma_t\| \setminus M_t,
\end{aligned}
\right.\]
where $M_{t}$ is the reduced boundary of $\{u_0(\cdot,t) = 1\}$, and $u_0(\cdot,t)$ is the bounded variation function given by the weak-$*$ limit of $u_\e(\cdot,t)$, as functions of bounded variation.
\end{lem}

We give here a brief explanation about the proof of the lemma above. By the \emph{Clearing-out Lemma} of Ilmanen \cite{Ilmanen} (see also Pisante-Punzo \cite[Lemma 4.1]{PP}), we know that $u_{\epsilon}$ converges locally uniformly to either $1$ or $-1$ as $\epsilon \downarrow 0$ at any point in $\{ |u_0^{t}| > \alpha \}$, where $0 < \alpha < 1$ is a constant. In addition, the proof of the integrality \cite{T03} of the varifold $\frac{1}{\sigma}V_t$ is based in an a.e.-graphical decomposition of the transition layers of $u_{\e}$, similarly to the elliptic counterpart in \cite{HT}. Hence, the argument about the oddness and evenness of the density in the proof of Theorem 1 in \cite{HT} can be carried out in the parabolic setting. We refer to \cite{TT} for the detailed proof of the parity of the density.\smallskip

Finally, we can describe the limits of the Brakke flow $\Sigma_t$ in $S^3$.

\begin{thm} \label{graphical}
As $t \to -\infty$, the varifold $\frac{1}{\sigma}V_t$ converges to some multiplicity one minimal torus in $S^3$, and its support converges graphically to this torus.
As $t \to +\infty$, the varifold $\frac{1}{\sigma}V_t$ subconverges to a multiplicity one equatorial sphere.
\end{thm}

\begin{proof}
Consider any sequence $\theta_i \uparrow \infty$, and the sequence of translated Brakke flows $\{\Sigma_t^{(i)} := \Sigma_{t+ \theta_i}\}_{t \geq 0}$. By Brakke’s compactness theorem and the uniform boundedness of areas, $(\Sigma_t^{(i)})_{t\geq 0}$ converges subsequentially to an integral Brakke flow with constant area. Therefore, this Brakke flow is supported on a stationary integral varifold $\frac{1}{\sigma} V_{+\infty}$. Similarly, $\Sigma_t$ subconverges, as $t \to -\infty$, to a stationary integral varifold $\frac{1}{\sigma} V_{-\infty}$.

By Lemma \ref{odd-even-density}, the associated $\mathbb{Z}_2$ chain of $\frac{1}{\sigma}V_t$ is $M_0^{t}$ is the reduced boundary of $\{u_0^{t} = 1 \}$, thus it has vanishing boundary. By White \cite{W09} Theorem $4.2$, we know that the associated $\mathbb{Z}_2$ chain of any subsequential limit varifold $\frac{1}{\sigma} V_{+\infty}$ or $\frac{1}{\sigma} V_{-\infty}$ has zero boundary.

We have shown the area estimate $\|\Sigma_t\|(S^3) \leq 2 \pi^2$. By \cite[Lemma 5.8]{CM}, it follows that any such limit $\frac{1}{\sigma} V_{-\infty}$ or $\frac{1}{\sigma} V_{+\infty}$ is either a multiplicity one equatorial sphere or a multiplicity one  Clifford torus. On the other hand, we have the inequality $\|\Sigma_{-t}\|(M) \geq 5 \pi \geq \|\Sigma_t\|(M)$ for every $t \geq 0$. Thus $\frac{1}{\sigma} \| V_{-\infty}\|(S^3) \geq 5\pi \geq \frac{1}{\sigma} \|V_{+\infty}\|(S^3)$, and we conclude that any subsequential limit $\frac{1}{\sigma} V_{-\infty}$ is a Clifford torus, and any subsequential limit $\frac{1}{\sigma} V_{+\infty}$ is a multiplicity one equatorial sphere. The convergence as $t \to -\infty$ and the graphical convergence of $\Sigma_t$ to the minimal torus, follows from \cite{CM}, by means of Brakke's local regularity for the mean curvature flow \cite{Brakke} (see also \cite{KT}), as the characterization of the backward limit allows us to obtain the smoothness of $\Sigma_t$ for sufficiently negative time.
\end{proof} 

As a final remark, we note that the energy $u_{\e_j}(\cdot,t)$ has a limit as $(\e_j,t) \to (0,-\infty)$ (where $\e_j \downarrow 0$ is a subsequence for which the associated varifolds converge):

\begin{lem}
Given $\delta>0$, there exist a positive integer $J_0=J_0(\delta)$ and $T_0=T_0(\delta)>0$ (independent of $\e$) such that
    \[0<E_{\e_j}(u_{\e_j}^{-\infty}) - E_{\e_j}(u_{\e_j}(\, \cdot \, , t)) < \delta, \qquad \text{for all} \ j\geq J_0 \ \text{and all} \ t<-T_0.\]
Consequently, 
\begin{align*}
\lim\limits_{(j,t) \to (+\infty,-\infty)} E_{\e_j}(u_{\e_j}(\cdot, t)) = 2\sigma \cdot 2\pi^2.
\end{align*}
\end{lem}

\begin{proof}
If not, since $\{u_{\e_j}(\cdot,t)\}$ are nonconstant gradient flows for $-E_{\e_j}$, there exists a $\delta > 0$, a subsequence $\{\e_{i}=\e_{j_i}\}$ of $\{\e_j\}$ and sequence $t_i \to -\infty$ such that $E_{\e_i}(u_{\e_i}^{-\infty}) - E_{\e_i}(u_{\e_i}(\cdot, t_i)) \geq \delta$ for all $i$.  We claim that this leads to a contradiction, for sufficiently large $i$.

By construction, $E_{\e}(u_{\e}^{-\infty}) \to 2\sigma(2\pi^2)$, so there exists $\e' > 0$ such that for any $\e < \e'$, it holds $E_{\e}(u_{\e}^{-\infty}) < 2\sigma(2\pi^2 + \frac{\delta}{6\sigma})$. On the other hand, $\Sigma_t$ (and the associated varifolds $\frac{1}{\sigma} V_t$) converge to a Clifford torus, thus there exists $t' \ll 0$ such that $\|\Sigma_{t'}\|(S^3) > 2\pi^2 - \frac{\delta}{6\sigma}$. For this fixed $t'$, using $\frac{1}{2\sigma} E_{\epsilon_i}(u_{\epsilon_i}(\cdot ,t')) \to \|\Sigma_{t'}\| (S^3)$ as $i \to +\infty$, we get some positive integer $I$ such that
    \[\left|\frac{1}{2\sigma} E_{\epsilon_i}(u_{\epsilon_i}(\cdot ,t')) - \|\Sigma_{t'}\| (S^3)\right| < \frac{\delta}{6\sigma}\]
for any $i \geq I$. Hence $E_{\e}(u_{\e}(\cdot,t')) > 2\sigma(\|\Sigma_{t'}\|(S^3) - \frac{\delta}{6\sigma})$.

Since $\e_i \downarrow 0$ and $t_i \to -\infty$, we may assume $I$ is such that $\e_i < \e'$ and $t_i < t'$ whenever $i \geq I$. Then (using again that the energy decreases along the flow),
    \begin{align*}
        E_{\e_i}(u_{\e_i}^{-\infty}) - \delta &\geq E_{\e_i}(u_{\e_i}(\cdot, t_i)) > E_{\e_i}(u_{\e_i}(\cdot, t')) \\
        & > 2\sigma\left(\|\Sigma_{t'}\|(S^3) - \frac{\delta}{6\sigma}\right)\\
        &> 2\sigma\left(2\pi^2 - \frac{\delta}{3\sigma}\right) = 2\sigma\left(2\pi^2+\frac{\delta}{6\sigma}\right) - \delta > E_{\e_i}(u_{\e_i}^{-\infty}) - \delta,
    \end{align*}
so we get a contradiction.
\end{proof}

\section{Symmetries and proof of main results} \label{sec:sym}

We use the notation introduced in Section \ref{sec:gradient} to describe the ancient solutions of the gradient flow of $-E_\e$ given by the results of Choi-Mantoulidis \cite{CM}, for the critical point $u_\e^{-\infty}$  which has the Clifford torus $T_c$ as its nodal set. Recall that, for $\e \in (0,\e_4)$ (as given by Lemma \ref{equatorial}), we denote by $\{\varphi_i\}_{i=1}^5$ an $L^2$-orthonormal basis for the eigenspaces of the linearized Allen-Cahn operator at $u_\e^{-\infty}$ corresponding to negative eigenvalues, where we assume $\varphi_1>0$. We will denote by $\mathcal{V}_1$ the first eigenspace, which is spanned by $\varphi_1$, and by $\mathcal{V}$ the eigenspace spanned by $\{\varphi_i\}_{i=2,3,4,5}$.

We also recall the solution map $\mathscr{S}_\e=\mathscr{S}\colon B_\eta(0)\subset \R^5 \to C^{2,\alpha}(S^3 \times \R)$, defined for some $\eta=\eta_\e>0$ (these solutions are also defined for all $t> 0$ by the long-time existence result described in the proof of Proposition \ref{prop_parabolicflow}). We will use the action of the isometry group of $S^3$ by pre-composition to study the limit in time of some of these solutions. As a consequence, we will establish the existence of two-parameter family of solutions to this parabolic equation joining $u_\e^{-\infty}$ to ground states, as well as Brakke flows joining the Clifford torus $T_c$ to equatorial spheres. This will conclude the proof of the main results.

\subsection{Isometries and invariance}

We introduce some isometries of $S^3$ that leave $u_\e^{-\infty}$ invariant. First, for $\theta \in \R$, let $\rho^\theta,\tau^\theta\colon S^3 \to S^3$ denote the rotations
		\begin{align*}
		    \rho^\theta(x) & = (x_1\cos \theta - x_2 \sin \theta , x_1 \sin \theta + x_2 \cos\theta, x_3,x_4)\\
		    \tau^\theta(x) & = (x_1,x_2, x_3\cos \theta - x_4 \sin \theta , x_3 \sin \theta + x_4 \cos\theta)
		\end{align*}
We also regard $\rho^\theta$ and $\tau^\theta$ naturally as isometries of $S^3 \times \R$, acting on the first factor, and as isometries of $B_\eta(0)\subset \R^5$ acting on the last $4$ coordinates, using the same notation for simplicity, as in $R(a_1,\ldots, a_5) = (a_1,R(a_2,a_3,a_4,a_5))$.\smallskip
	
Observe that $u_\e^{-\infty}$ is invariant by both $\rho^\theta$ and $\tau^\theta$. Thus, these rotations leave $\varphi_1$ invariant and act linearly and isometrically on the left on the eigenspace $\mathcal{V}$ by pre-composition, that is
    \[\varphi \mapsto \varphi \circ \rho^{-\theta} \quad \text{and} \quad \varphi \mapsto \varphi \circ \tau^{-\theta}.\]
We can describe this action as rotations in $\mathcal{V}$. In fact, this describes a representation of $\mathrm{SO}(2) \times \mathrm{SO}(2)$ on $\mathcal{V}$ taking values in $\mathrm{SO}(\mathcal{V}) \simeq \mathrm{SO}(4)$.  We will need the following observation:

\begin{lem} \label{representation}
Let $\e \in (0,\e_4)$, where $\e_4$ is given by Lemma \ref{equatorial}. The kernel of the representation described above is trivial. That is, if $(\theta,\zeta) \in \R^2$ is such that $\varphi \in \mathcal{V}\mapsto \varphi \circ (\rho^{-\theta} \circ \tau^{-\zeta})$ is the identity map, then $\frac{\theta}{2\pi},\frac{\zeta}{2\pi} \in \mathbb{Z}$. 
\end{lem}

We postpone the proof of the lemma above for later. As a consequence, we can choose the eigenfunctions $\varphi_i$ in a way that $\{\rho^\theta\}$ acts on $\{\varphi_2,\varphi_3\}$ by rotation and fixes $\{\varphi_4,\varphi_5\}$ pointwise, while $\{\tau^\theta\}$ acts on the latter by rotations and fixes $\varphi_2,\varphi_3$. Concretely, we may assume
	\begin{equation} \label{rotations}
	    \left\{ \begin{array}{rcl} \varphi_2 \circ \rho^{-\theta} & =& \ \ \left(\cos\theta\right) \varphi_2 + \left(\sin\theta\right) \varphi_3 \\ \varphi_3 \circ \rho^{-\theta} & = & \left(-\sin\theta\right) \varphi_2 + \left(\cos\theta\right) \varphi_3 \\ \varphi_i \circ \rho^{-\theta} &=& \varphi_i, \qquad i=4,5, \end{array} \right. \quad \text{and} \quad \left\{ \begin{array}{rcl} \varphi_4 \circ \tau^{-\theta} & =& \ \ \left(\cos\theta\right) \varphi_4 + \left(\sin\theta\right) \varphi_5 \\ \varphi_5 \circ \tau^{-\theta} & = & \left(-\sin\theta\right) \varphi_4 + \left(\cos\theta\right) \varphi_5 \\ \varphi_i \circ \tau^{-\theta} &=& \varphi_i,\qquad i=2,3. \end{array} \right.
	  \end{equation} 
This follows from the lemma above, as the injective image of this representation in $\mathrm{SO}(4)$ is a closed $2$-torus subgroup. Then, it suffices to note that any maximal torus in $\mathrm{SO}(4)$ is conjugated to the standard torus $\mathrm{SO}(2) \times \mathrm{SO}(2) \subset \mathrm{SO}(4)$. Up to changing the basis of $\mathcal{V}$, we obtain \eqref{rotations}. We emphasize that the minus sign in the angles are chosen so that we have an action on the left, similarly to the actions  on $S^3$, $S^3\times \R$ and $B_\eta(0)$.\smallskip

The uniqueness of solutions of \eqref{PAC} that converge back to $u_\e^{-\infty}$ can be used to establish the following equivariance property:
	
\begin{lem} \label{lem:equivariance}
The solution map $\mathscr{S}$ satisfies
	\[\mathscr{S}(a)(\rho^{-\theta} (x),t) = \mathscr{S}(\rho^{\theta} (a))(x,t), \quad \text{for any} \quad (x,t) \in S^3 \times \R,\]
for every $\theta \in \R$, and similarly for $\tau^\theta$.
\end{lem}

\begin{proof}
Since $u_\e^{-\infty}$ is invariant by $\rho^{-\theta}$, the solution $\mathscr{S}(a)(\rho^{-\theta} (x),t)$ still converges to $u_\e^{-\infty}$ as $t \to -\infty$. By the uniqueness of such solutions, it suffices to check that the $L^2$-projections of $\mathscr{S}(a)(\rho^{-\theta} (x), 0) - u_\e^{-\infty}(x)$ and $\mathscr{S}(\rho^{\theta} (a))(x, 0) - u_\e^{-\infty}(x)$ onto $\mathcal{V}_1\oplus \mathcal{V}$ coincide. By the invariance of $u_\e^{-\infty}$ and $\varphi_1$, a direct computation using \eqref{rotations} shows that
	\begin{align*}
		 \int_{S^3} \left(\mathscr{S}(a)(\rho^{-\theta} (x), 0) - u_\e^{-\infty}(x)\right)\varphi_i(x)\,d\mu_g(x) & = \int_{S^3} \left( \mathscr{S}(a)(x,0) - u_\e^{-\infty}(x)\right) \varphi_i(x)\,d\mu_g(x) = a_i\\
		&= \int_{S^3} \left( \mathscr{S}(\rho^{\theta} (a))(x,0) - u_\e^{-\infty}(x)\right) \varphi_i(x)\,d\mu_g(x)\\
	\end{align*}
for $i=1,4,5$, while
	\begin{align*}
		& \int_{S^3} \left(\mathscr{S}(a)(\rho^{-\theta} (x), 0) - u_\e^{-\infty}(x)\right)\varphi_2(x)\,d\mu_g(x) \\
		& \qquad = (\cos\theta) \int_{S^3} \left( \mathscr{S}(a)(x,0) - u_\e^{-\infty}(x)\right) \varphi_2(x)\,d\mu_g(x) \\
		& \qquad \qquad - (\sin\theta) \int_{S^3} \left( \mathscr{S}(a)(x,0) - u_\e^{-\infty}(x)\right) \varphi_3(x)\,d\mu_g(x)\\
		& \qquad = (\cos\theta)a_2 - (\sin\theta)a_3 = \int_{S^3} \left(\mathscr{S}(\rho^{\theta} (a))(x, 0) - u_\e^{-\infty}(x)\right) \varphi_2(x)\,d\mu_g(x)
	\end{align*}
and similarly for the projection of $\left(\mathscr{S}(a)(\rho^{-\theta} (x), 0) - u_\e^{-\infty}(x)\right)$ onto $\varphi_3$. The proof for $\tau^\theta$ is similar.
\end{proof}

We also consider the isometry $s(x) = (x_3,x_4,x_1,x_2)$ acting on $S^3$. From the construction of $u_\e^{-\infty}$ (see \cite{CGGM} or \cite{H}), this isometry satisfies $u_\e^{-\infty}\circ s = -u_\e^{-\infty}$, hence it preserves the linearized Allen-Cahn operator at $u_\e^{-\infty}$. It follows that the linear operator $ \varphi \mapsto -\varphi\circ s$ defined on $\mathcal{V}_1 \oplus \mathcal{V}$ maps $\varphi_1$  to $(-\varphi_1)$ and, possibly after changing the basis of $\mathcal{V}$ by rotations, it maps $\varphi_2, \varphi_3, \varphi_4, \varphi_5$ into $(-\varphi_4)$, $(-\varphi_5)$, $(-\varphi_2)$ and $(-\varphi_3)$, respectively. Arguing as in the proof of Lemma \ref{lem:equivariance}, we see that
	\begin{equation}
	    -\mathscr{S}(a)(s(x), t) = \mathscr{S}(-s(a))(x,t),
	\end{equation}
where write $s(a)$ to mean $(a_1,s(a_2,a_3,a_4,a_5))=(a_1,a_4,a_5,a_2,a_3)$.

\begin{rmk}
The orbits of points $b \in B_\eta(0)$ by the action by $\rho^\theta$ and $\tau^\theta$ are the tori
	\[\{ a \in B_\eta(0) \mid a_1=b_1, a_2^2 + a_3^2 = b_2^2 + b_3^2, a_4^2 + a_5^2 = b_4^2 + b_5^2 \},\]
which degenerate to circles (or points) when $(b_2,b_3) = 0$ or $(b_4,b_5)=0$. Note that if the corresponding solution $\mathscr{S}(a)$ converges to a constant $c$ as $t \to +\infty$, for some point $a$ in this orbit, then the solutions corresponding to any point in the same orbit converge to $c$, as these two solutions to \eqref{PAC} differ by pre-composition with an isometry of the sphere.

For each small $\e>0$, we fix $r=r(\e) \in (0,\eta)$ depending continuously on $\e$. We will be particularly interested in the orbit
	\begin{equation} \label{orbit}
	    \mathcal{O}_\e = r\, \mathcal{O}, \quad \text{where} \ \mathcal{O} = \left\{p \in S^4 \mid p_1 = 0, p_2^2 + p_3^2 = \frac{1}{2} = p_4^2 + p_5^2\right\} \subset \R^5,
	  \end{equation}
which is invariant by the symmetry $s$. Note that this orbit may possibly change with $\e>0$, but only by a dilation. Geometrically, each direction in $\mathcal{O}$ corresponds to some deformation of the Clifford torus that decreases its area (this normalization will be useful later, when we study the forward limit of $\mathscr{S}(a)$ and the corresponding Brakke flow).
\end{rmk}

Finally, we consider reflections $r_v \colon S^3 \to S^3$ given by $r_v(x) = x - 2\langle x,v \rangle v$, for unit vectors $v \in \R^4$, and the corresponding action on (the four last coordinates of) $B_\eta(0)$. If $v_1=0=v_2$ (respectively if $v_3=0=v_4$) then $u_\e^{-\infty} \circ r_v = u_\e^{-\infty}$ and $\varphi_1 \circ r_v = \varphi_1$. Moreover, $\varphi \mapsto \varphi\circ r_v$ is a linear involution on $\mathcal{V}$ and it commutes with $\tau^\theta$ (respectively, with $\rho^\theta$), so it defines an operator on $\mathrm{span}\{\varphi_2,\varphi_3\}$ (respectively, on $\mathrm{span}\{\varphi_4,\varphi_5\}$). Using basic properties of linear involutions, one deduces:

\begin{lem} \label{lem:reflection}
Suppose $v \in S^3$ and $v_3=0=v_4$ (respectively, $v_1=0=v_2$). Denote by $E_v^{\pm}$ the $(\pm 1)$-eigenspaces of $r_v$ in the space $E$ spanned by $\varphi_2,\varphi_3$  (respectively, $\varphi_4,\varphi_5$). Then
\begin{enumerate}
	\item[(i)] $\dim E_v^{+} = 1 = \dim E_v^-$.
	\item[(ii)] For any $\varphi \in E$, there is such a unit vector $v$ satisfying $\varphi \circ r_v = \varphi$. 
	\item[(iii)] $E_{\rho^{\pi/2}(v)}^{\pm}  = E_v^{\mp}$ (respectively, $E_{\tau^{\pi/2}(v)}^{\pm}  = E_v^{\mp}$).
\end{enumerate} 
\end{lem}

\begin{proof}
We assume $v_3=0=v_4$, the other case being analogous. Denote by $E$ the space spanned by $\varphi_2$ and $\varphi_3$, the fixed space of $\tau^\theta$ in $\mathcal{V}$. The only possible eigenvalues of $\varphi \mapsto \varphi \circ r_v$ are $\pm 1$. We cannot have $E_v^{-}= \{0\}$ for all such $v$, otherwise any $\varphi \in E$ would satisfy $\varphi \circ r_v = \varphi$. This leads to a contradiction, as we can find such $u,v \in S^3$ so that $r_u \circ r_v = \rho^{-\pi/2}$, implying $\varphi \circ \rho^{-\pi/2} = \varphi$, which is impossible for nonzero $\varphi \in E$.

This proves that, for some $v \in S^3$ with $v_3=0=v_4$, the eigenspace $E_v^{-}$ is nontrivial. This conclusion then holds for every such $v$, since $r_{\rho^\theta (v)} \circ \rho^\theta = \rho^\theta \circ r_v$. Similarly, we see that $E_v^{+}$ is nontrivial, so these eigenspaces must be one-dimensional. Finally, one readily concludes (ii) and (iii) from the relation $r_{\rho^\theta (v)} \circ \rho^\theta = \rho^\theta \circ r_v$ and the fact that $E_v^+$ and $E_v^-$ are orthogonal.
\end{proof}

In order to study the backward and forward limits of the Brakke flow constructed from sequences of solutions to \eqref{PAC}, it will be useful to describe $E_v^\pm$ more explicitly. For each $\epsilon \in (0,\e_4)$, after possibly rotating $\varphi_2$ and $\varphi_3$, we can assume that $\varphi_2$ spans the $+1$ eigenspace of $r_{e_2}$, that is $E_{e_2}^{+}$. Then by Lemma \ref{lem:reflection}, the reflection $r_{e_1}$ reverses the sign of $\varphi_2$,  $r_{e_2}$ reverses the sign of $\varphi_3$, and $\varphi_3$ is preserved by $r_{e_1}$. The eigenfunctions $\varphi_2$ and $\varphi_3$ we get satisfy the following equations
\begin{align*}
&\varphi_2(x_1,x_2,x_3,x_4) = \varphi_2(x_1,-x_2,x_3,x_4), \\
&\varphi_2(x_1,x_2,x_3,x_4) = -\varphi_2(-x_1,x_2,x_3,x_4), \\
&\varphi_3(x_1,x_2,x_3,x_4) = -\varphi_3(x_1,-x_2,x_3,x_4), \\
&\varphi_3(x_1,x_2,x_3,x_4) = \varphi_3(-x_1,x_2,x_3,x_4).
\end{align*}

Moreover, we can assume that $\varphi_4 = \varphi_2 \circ s$ and $\varphi_5 = \varphi_3 \circ s$ to get a similar property with respect to reflections on $x_3$ and $x_4$. This choice is compatible with the action by rotations $\rho^{\theta}$ and $\tau^{\theta}$ on the space generated by these eigenfunctions, so that \eqref{rotations} remains valid. Furthermore, using this choice, we see that $\mathscr{S}$ is equivariant with respect to reflections that preserve $T_c$, namely
    \begin{equation}
        \mathscr{S}(a)(r_v(x),t) = \mathscr{S}(r_{v}(a))(x,t)
    \end{equation}
for any unit vector $v$ with $v_1=0=v_2$ or $v_3=0=v_4$. In fact, the computation is straightforward for $v=e_i$, $i=1,2,3,4$, and the general case follows from
\begin{align*}
\mathscr{S}(a)(r_{\rho^\theta(e_1)}(x),t) &= \mathscr{S}(a)(\rho^{\theta} \circ r_{e_1} \circ \rho^{-\theta}(x),t) \\
&= \mathscr{S}(\rho^{-\theta} \circ r_{e_1} \circ \rho^{\theta}(a))(x,t) \\
&= \mathscr{S}(r_{\rho^\theta(e_1)}(a))(x,t),
\end{align*}
and similarly for $\tau^\theta(e_3)$.

\begin{proof}[Proof of Lemma \ref{representation}]
Suppose that $\theta,\zeta \in \R$ satisfy the hypothesis in the statement of the lemma. By Propositions \ref{prop_parabolicflow} and \ref{prop2}, there exists $a=(a_1,a_2,a_3,a_4,a_5) \in \partial B_r(0)$ such that $\mathscr{S}(a)(\cdot,t)$ converges in $W^{1,2}$ to a ground state solution $w$ of \eqref{AC}, as $t \to +\infty$. Let $y \in S^3$ be such that $S^3 \cap \{w=0\} = y^\perp = \{x \in S^3 \mid \langle x,y \rangle =0\}$. We have either $y_1^2 + y_2^2 \neq 0$ or $y_3^2 + y_4^2 \neq 0$. Suppose the former holds; the other case can be argued similarly.

If we write $\varphi = a_1 \varphi_1 + \ldots + a_5 \varphi_5$, then $\varphi \circ (\rho^{-\theta} \circ \tau^{-\zeta}) = \varphi$, by assumption. By the uniqueness of solutions to \eqref{PAC} that converge backward to $u_\e^{-\infty}$, computing projections as in the proof of Lemma \ref{lem:equivariance} (without using the formulas in \eqref{rotations}), we see that $\mathscr{S}(a)$ and hence $w$ are invariant by $(\rho^{-\theta} \circ \tau^{-\zeta})$. In particular, the nodal set of $w$ is preserved by this composition of rotations, so $(\rho^{-\theta} \circ \tau^{-\zeta})(y) = \pm y$. Since $y_1^2+y_2^2 \neq 0$ and $\tau^{-\zeta}$ does not change $(y_1,y_2)$, this implies $\frac{\theta}{2\pi} \in \mathbb{Z}$.

Now consider the solution $\mathscr{S}(-s(a)) = - \mathscr{S}(a)\circ s$, which converges, as $t \to +\infty$, to a ground state solution with nodal set $s(y^\perp) = z^\perp$, where $z=s(y) \in S^3$ is such that $z_3^2 + z_4^2 \neq 0$. This solution is also invariant by $(\rho^{-\theta} \circ \tau^{-\zeta})$, by assumption, allowing us to conclude  $\frac{\zeta}{2\pi} \in \mathbb{Z}$.
\end{proof}

We conclude this subsection by noting that the equivariance of the solution map $\mathscr{S}$ implies the equivariance of the varifolds $V_{\e,\mathscr{S}(a)(\cdot,t)}$ with respect to $a$. This follows from the fact that, for any isometry $P$ of $S^3$ and any $C^1$ function $u$ defined on $S^3$, the pushforward $P_\#$ by the isometry $P$ satisfies
    \[P_{\#}V_{\e,u} = V_{\e, u \circ P^{-1}}.\]
This relation follows from a change of variables; we refer to \cite{H} for the proof. Noting also that $V_{\e,u} = V_{\e,-u}$, we obtain
    \begin{align*}
        V_{\e,\mathscr{S}(\rho^\theta(a))(\cdot,t)} & = (\rho^\theta)_{\#}V_{\e,\mathscr{S}(a)(\cdot,t)}, \qquad &
        V_{\e,\mathscr{S}(-s(a))(\cdot,t)} & = s_{\#}V_{\e,\mathscr{S}(a)(\cdot,t)},\\
          V_{\e,\mathscr{S}(\tau^\theta(a))(\cdot,t)} & = (\tau^\theta)_{\#}V_{\e,\mathscr{S}(a)(\cdot,t)}, \qquad & V_{\e,\mathscr{S}(r_v(a))(\cdot,t)} & = (r_v)_{\#}V_{\e,\mathscr{S}(a)(\cdot,t)},
    \end{align*}
for every $\theta \in \R$, and for every $v \in S^3$ such that $v_1=0=v_2$ or $v_3=0=v_4$. Observe that this also implies similar relations for the varifolds associated to any translated solution $\mathscr{S}(a)(x,t+t_0)$, for $t_0 \in \R$, and shows the invariance of the solutions and their limit as $t \to +\infty$ as well as the limits of these varifolds as $\e \downarrow 0$  whenever $a$ is fixed by the respective isometry, provided such limits exist.

\subsection{Backward limit}

As mentioned in Section \ref{sec:main}, the arguments described so far only guarantee that the limit Brakke flow extracted from a sequence of translated solutions of \eqref{PAC} converges back to \emph{some} Clifford torus in $S^3$ (provided these solutions join $u_\e^{-\infty}$ to some nonconstant critical point of $E_\e$). Our main goal in this subsection is to use the symmetries $r_v$ and $s$ and a continuity argument to show that any Brakke flow constructed from $\mathscr{S}(a)$ has $T_c$ as its backward limit whenever $a$ lies in the orbit $\mathcal{O}_\e$ described in \eqref{orbit}. This will conclude the proof of Theorem \ref{thm1}. \medskip

Let $v=\frac{e_1 - e_2}{\sqrt{2}}\in S^3$. Since $v = \rho^{-\frac{\pi}{4}} ({e_1})$, we have $r_{v} = \rho^{-\frac{\pi}{4}} \circ r_{e_1} \circ \rho^{\frac{\pi}{4}}$. Hence we know that $r_v$ is the reflection $(x_1,x_2,x_3,x_4) \mapsto (x_2,x_1,x_3,x_4)$. Let also $w = s(v) = \tau^{-\frac{\pi}{4}}(e_3)$, and note that $r_w(x_1,x_2,x_3,x_4) = (x_1,x_2,x_4,x_3)$.

For $b = r(\e)\cdot(0, \frac{1}{2},  \frac{1}{2}, -\frac{1}{2}, -\frac{1}{2}) \in \mathcal{O}_\e$, we have $r_{v}(b) = b$, $r_{w}(b) = b$, and $s(b) = -b$. Hence the solution $\mathscr{S}(b)$ is invariant by the reflections $r_{v}$ and $r_{w}$, and it changes its sign under the isometry $s$. Consequently, the limit Brakke flow and its backward limit are also invariant under the reflections $r_{v}, r_{w}$ and the isometry $s$.

\begin{lem}
There are only three Clifford tori in $S^3$ that are invariant under the reflections $r_{\frac{e_1 - e_2}{\sqrt{2}}}$ and $r_{\frac{e_3 - e_4}{\sqrt{2}}}$, and the isometry $s$. These are $T_c=\{x \in S^3 \mid x_1^2+x_2^2 = x_3^2 + x_4^2\}$, $T_+=\{x \in S^3 \mid x_1 x_2 = x_3 x_4\}$, and $T_-=\{x \in S^3 \mid x_1 x_2=-x_3 x_4\}$.
\end{lem}

We only sketch the proof here. Disregarding orientations, any Clifford torus is uniquely determined by the choice of a $2$-plane $P$ in $\R^4$, so that it is explicitly given by
    \[\{ x \in S^3 \mid \|P(x)\|^2 = \|P^\perp(x)\|^2\},\]
where we use the same notation $P$ for the orthogonal projection $\R^4 \to P$. One checks that if this is preserved by a reflection $r_y$, for some unit vector $y$, then $y\in P \cup P^\perp$.

This means that any torus preserved by the reflections $r_v$ and $r_w$ described above are such that $v,w \in P \cup P^\perp$. Unless $v,w \in P$ (in which case $P$ is spanned by $\{v,w\}$), the plane $P$ is determined by the choice of a unit vector orthogonal to $v$ and $w$. By writing their explicit equations and checking which of these tori are invariant by the isometry $s$, one concludes the proof.

\begin{prop}
Let $a \in \mathcal{O}_\e$, as defined in \eqref{orbit}, and let $u_\e$ by the solution of \eqref{PAC} given by the time translation of $\mathscr{S}(a)$ such that $u_\e(\cdot,0)$ has energy $5\pi$. Then any (subsequential) limit Brakke flow $\Sigma_t$ obtained from $u_\e$ as $\e \downarrow 0$ converges to $T_c$ as $t \to -\infty$.
\end{prop}

\begin{proof}
We will prove the result for $q=r(\e)\cdot (0, \frac{1}{2},  \frac{1}{2}, -\frac{1}{2}, -\frac{1}{2})$. Since any $a \in \mathcal{O}_\e$ is obtained from $b$ by rotations that preserve $T_c$, the claimed result will then follow from Lemma \ref{lem:equivariance}.

For a $2$-varifold (or a surface) $V$ in $S^3$, we will write $V(f)$ to mean the integral of a continuous function $f$ on the Grassmannian $G_2(S^3)$ with respect to $V$ (or the multiplicity 1 varifold induced by such surface). Since $T_c$ and $T_{\pm}$ are distinct tori, we can pick a smooth real-valued function $f$ defined on $G_2(S^3)$ such that 
    \[-1 = T_+(f) = T_-(f) < 0 < T_c(f) = 1.\]

By Theorem \ref{graphical} and the previous Lemma, we know that $\Sigma_t$ converges, as $t \to -\infty$, to either $T_c$, $T_-$ or $T_+$. Suppose, by contradiction, that it converges to $T_+$. 
Pick a sequence $t_j \to -\infty$ such that $|\Sigma_{t_j}(f) - (-1)|<1/6$. We can find a sequence $\{\e_j\}$ of positive numbers such that $\e_j \downarrow 0$,
\begin{itemize}
    \item the varifold distance between $\frac{1}{\sigma}V_{\e_j,t_j}$ and $\frac{1}{\sigma}V_{t_j}$ is bounded above by $1/j$ (where $\Sigma_{t_j}$ is the weight measure of $\frac{1}{\sigma}V_{t_j}$), and
    \item $\left|\frac{1}{\sigma}V_{\e_j,t_j}(f)-(-1)\right|<1/3$.
\end{itemize}
Since the varifold $V_\e^{-\infty}$ associated to $u_\e^{-\infty}$ converges to $\sigma \cdot T_c$ as $\e \downarrow 0$, we can also assume
    \[\left|\frac{1}{\sigma} V_{\e_j}^{-\infty}(f) - 1\right| < \frac{1}{3}.\]

Note that $V_{\e_j,t}$ varies continuously in the varifold topology for $t \in [-\infty,0]$, where we write $V_{\e,-\infty} = V^{-\infty}_{\e}$. Thus, $t \in [-\infty,0] \mapsto \frac{1}{\sigma}V_{\e_j,t}(f)$ is a continuous function such that
    \[\frac{1}{\sigma}V_{\e_j,t_j}(f) < -\frac{2}{3} < \frac{2}{3} < \frac{1}{\sigma}V_{\e_j}^{-\infty}(f), \]
so there exists $s_j \in (-\infty,t_j) \subset (-\infty,0)$ such that $\frac{1}{\sigma}V_{\e_j,s_j}(f) = 0$.\medskip

\noindent\textbf{Claim.} After possibly passing to a (non relabeled) subsequence, $\{\frac{1}{\sigma}V_{\e_j,t+s_j}\}_{t}$ converges,as $j\to+\infty$, to a Brakke flow which is integral for a.e. $t$, and which has constant area $2\pi^2$, for every $t \leq 0$. Consequently, this Brakke flow is supported on some Clifford torus $T'$, and $V_{\e_j,s_j} \to T'$ as varifolds.\smallskip

To prove this claim, let $\tilde u_j(x,t) = u_{\e_j}(x,t+s_j)$. Then $\tilde u_j$ are solutions to \eqref{PAC} such that $\tilde u_j(\cdot,t) \to u_{\e_j}^{-\infty}$ as $t \to -\infty$, and that have the same limit as $u_{\e_j}(\cdot,t)$ when $t \to +\infty$. One checks that the sequence $\{\tilde u_j\}$ satisfies the hypotheses of Lemma \ref{bounds}, and
    \[E_{\e_j}(\tilde u_j(\cdot, t)) \leq E_{\e_j}(u_{\e_j}^{-\infty}), \quad \text{for all} \quad t \in \R.\]
Hence, after possibly passing to a subsequence, $\{\frac{1}{\sigma}V_{\e_j,\tilde u_j(\cdot, t)}=\frac{1}{\sigma}V_{\e_j,t+s_j}\}_{t \in \R}$ converges, as $j \to +\infty$, to a $2$-dimensional Brakke flow $\{\Gamma_t\}_t$  which is integral for almost every $t \in \R$, and such that
    \begin{align*}
        \|\Gamma_t\|(S^3) & = \lim_{j \to +\infty}\frac{1}{2\sigma}E_{\e_j}(\tilde u_j(\cdot,t)) \leq \lim_{j \to +\infty}\frac{1}{2\sigma} E_{\e_j}(u_{\e_j}^{-\infty}) = 2\pi^2
    \end{align*}
for all $t \in \R$. Furthermore,
    \begin{align*}
        2\pi^2 &= \|T_+\|(S^3) = \lim_{j \to +\infty} \frac{1}{2\sigma}E_{\e_j}(u_{\e_j}(\cdot, t_j)) \leq \lim_{j \to +\infty} \frac{1}{2\sigma}E_{\e_j}(u_{\e_j}(\cdot, s_j)) \\
        & \leq \lim_{j \to +\infty} \frac{1}{2\sigma} E_{\e_j}(\tilde u_j(\cdot, t)) = \lim_{j \to +\infty} \frac{1}{\sigma}\|V_{\e_j,\tilde u_j(\cdot,t)}\|(S^3) = \|\Gamma_t\|(S^3)
    \end{align*}
for all $t\leq 0$. This implies $\|\Gamma_t\|(S^3) = 2\pi^2$ for every $t\leq 0$. We conclude that $\{\Gamma_t\}_{t \in \R}$ must be supported on a Clifford torus $T'$. Since $\Gamma_t$ is integral for a.e. $t$, this proves the claim.\medskip

We can now obtain a contradiction by recalling that the choice of $a$ implies that $\mathscr{S}(a)$ is invariant by the reflections with respect to $v=\frac{e_1-e_2}{\sqrt{2}}$ and $w = \frac{e_3-e_4}{\sqrt{2}}$, as well as the isometry $s$. Consequently, the translated solutions $u_{\e_j}$ and $\tilde u_j$, the limit Brakke flows $\Sigma_t$ and $\Gamma_t$, and the torus $T'$ given by the Claim above are invariant by the same isometries. But this contradicts
    \begin{align*}
        T'(f) & = \lim_{j \to +\infty}V_{\e_j,s_j}(f) = 0 \notin \{T_c(f),T_{\pm}(f)\},
    \end{align*}
where we used $\frac{1}{\sigma} V_{\e_j,s_j} \to \Gamma_0 = T'$. Similarly, $\Sigma_{t}$ does not converge to $T_-$ as $t \to -\infty$.
\end{proof}

This Proposition establishes the existence of solutions of \eqref{PAC}, for small $\e>0$, such that any limit Brakke flow converges back in time to the (same) Clifford torus $T_c$, thus proving Theorem \ref{thm1}.

\subsection{Forward limit}

We are in position to study the forward limit of solutions $\mathscr{S}(a)$, for $a$ in the orbit $\mathcal{O}_\e$ described in \eqref{orbit}, and of their limit Brakke flow. The key observation is that the invariance by the isometry $s$, which holds for some points in $\mathcal{O}_\e$, allows one to obtain information about the nodal set of the limit critical point. By arguing as in the previous subsection, we can then describe the convergence and the forward limit of the Brakke flows constructed as limits of these solutions.

For the next result, we recall the small positive $\e_4>0$ given by Lemma \ref{equatorial}, which ensures the convergence of certain solutions of \ref{PAC} to ground states.

\begin{prop} \label{forward}
Let $\e \in (0,\e_4)$. For each $p \in \mathcal{O}$, there is $y_p \in S^3$ (possibly depending on $\e$) such that
\begin{enumerate}
    \item[(i)] Let $u_p$ be the unique ground state solution of \eqref{AC} such that $\{u_p = 0\}$ is the equatorial sphere $S^3 \cap y_p^\perp$, and $u_p$ is positive in $\{x \in S^3 \mid \langle x, y_p\rangle>0\}$. Then
        \[\lim_{t \to +\infty}\|\mathscr{S}(r(\e)p)(\cdot,t)-u_p\|_{W^{1,2}(S^3)}=0.\]
    \item[(ii)] The map $p \mapsto y_p$ is injective, continuous, and odd.
\end{enumerate}
Moreover, we can find a sequence $\e_j \downarrow 0$ such that, for every $p \in \mathcal{O}$, the vector $y_p$ is independent of $j$ and, if $\Sigma_t$ is the Brakke flow in $S^3$ obtained from (a time-translation of) $\mathscr{S}(r(\e)p)$, as described in Section \ref{sec:main}, then $\Sigma_t$ converges to $y_p^\perp \cap S^3$, as $t \to +\infty$, in the varifold sense.
\end{prop}

First, we note that that solutions to $\eqref{PAC}$ that correspond to initial conditions in $\mathcal{O}_\e$ always join $u^{-\infty}_\e$ to ground states.

\begin{lem} \label{lem:nonconst}
Let $\e \in (0,\e_4)$, as given by Lemma \ref{equatorial}. For any $a \in \mathcal{O}_\e$, the solutions $\mathscr{S}(a)(\cdot, t)$ do not subsequentially converge (in the $W^{1,2}$ or H\"older norm) to the constants $\pm 1$ as $t \to +\infty$. Consequently, $\mathscr{S}(a)(\cdot,t)$ converges, in $W^{1,2}$ norm, to a ground state solution.
\end{lem}

\begin{proof}
Suppose some $\mathscr{S}(a)(\cdot,t)$ converges to the constant critical point $c \in \{-1,+1\}$ along a subsequence $t_j \to +\infty$. Then $\mathscr{S}(a)(s(x),t) = -\mathscr{S}(-s(a))(x,t)$ is a solution of \eqref{PAC} that converges subsequentially to $c$ as $t=t_j \to + \infty$. Consequently, $\mathscr{S}(-s(a))$ also solves \eqref{PAC} and converges to $-c$ as $t_j \to +\infty$. On the other hand, $-s(a) \in \mathcal{O}_\e$, so we obtain a contradiction, as $\mathscr{S}(a)$ and $\mathscr{S}(-s(a))$ differ by an isometry by Lemma \ref{lem:equivariance} (see also the Remark above about the orbits of $\rho^\theta$ and $\tau^\theta$). The last conclusion follows from Proposition \ref{prop2}.
\end{proof}

\begin{proof}[Proof of Proposition \ref{forward}]
Let $\e \in (0,\e_4)$. For every $p \in \mathcal{O}$, we let $u_{p,\e}$ be the limit of $\mathscr{S}(r(\e)p)$, which is a ground state solution, by Lemma \ref{lem:nonconst}. This means we can find $y_{p,\e} \in S^3$ such that $\{u_{p,\e}>0\} = \{x \in S^3 \mid \langle x, y_{p,\e}\rangle>0\}$, and $u_{p,\e}$ vanishes precisely at $y_{p,\e}^{\perp}\cap S^3$. This proves (i). By Lemma \ref{lem:equivariance}, we have
    \[u_{(\rho^{\theta_1} \circ \tau^{\theta_2})(p),\e} = u_{p,\e} \circ \rho^{-\theta_1} \circ \tau^{-\theta_2},\]
for every $\theta_1,\theta_2 \in \R$. This proves the continuity of $p \mapsto y_{p,\e}, u_{p,\e}$.

Again, we pick $q = (0,\frac{1}{2},\frac{1}{2},-\frac{1}{2},-\frac{1}{2}) \in \mathcal{O}$, so that $s(q)=-q$. We also consider the eigenfunction $\varphi = \frac{r(\e)}{2}(\varphi_2 + \varphi_3 - \varphi_4 - \varphi_5)$. As noted in the previous subsection, $v=\frac{e_1-e_2}{\sqrt{2}}$ and $w=s(v) = \frac{e_3-e_4}{\sqrt{2}}$ are such that $\varphi \circ r_v = \varphi = \varphi \circ r_w$. The equivariance of the solution map implies $u_{q,\e}\circ s = -u_{q,\e}$ and thus $y_{q,\e}^\perp$ is invariant by the isometry $s$. This shows that $s(y_{q,\e}) = \pm y_{q,\e}$ and $y_1^2 + y_2^2 = y_3^2 + y_4^2 \neq 0$, where $y_{q,\e}=(y_1,y_2,y_3,y_4)$. We also have $u_{q,\e} \circ r_v = u_{q,\e} = u_{q,\e} \circ r_w$, and ${y_{q,\e}}^\perp$ is invariant by both $r_v$ and $r_w$.

We can now prove the injectivity and oddness. Observe that if $\theta_1, \theta_2 \in \R$, and if the rotation $R=(\rho^{-\theta_1} \circ \tau^{-\theta_2})$ preserves $y_{q,\e}^\perp$, then $R(y_{q,\e}) = \pm y_{q,\e}$. Since $(y_1,y_2)$ and $(y_3,y_4)$ are nonzero, this can only happen if $R$ is either the identity or the antipodal map in $S^3$. Now any $p \in \mathcal{O}$ distinct from $q$ is of the form $p=(\rho^{\theta_1} \circ \tau^{\theta_2})(q)$, for some $\theta_1,\theta_2$ not both in $2\pi \mathbb{Z}$, which implies, by Lemma \ref{lem:equivariance}, that $u_{p,\e}$ is either $- u_{q,\e}$, in the case $p=-q$ (when $\theta_1$ and $\theta_2$ are odd multiples of $\pi$), or a critical point of $E_\e$ having a different equator as its nodal set. This shows that this limit is injective and odd with respect to $p$, and concludes the proof of (ii).

To prove the last statement, note that there are precisely two equators that are invariant by $r_v$ and $s$, namely
    \[\{x \in S^3 \mid x_1 + x_2 + x_3 + x_4 = 0\} \quad \text{and} \quad \{x \in S^3 \mid x_1 + x_2 -x_3 - x_4 =0\}.\]
Hence, for every $\e \in (0,\e_4)$, we have $y_{q,\e} \in \left\{\pm \left( \frac{1}{2}, \frac{1}{2}, \frac{1}{2}, \frac{1}{2} \right),\pm \left( \frac{1}{2}, \frac{1}{2}, -\frac{1}{2}, -\frac{1}{2} \right)\right\}$. This means we can find a sequence $\e_j \downarrow 0$ such that $y_{q,\e_j}=y$ is constant, and $\mathscr{S}(r(\e_j)q)$ converge to ground states $u_{q,\e_j}$ having the same nodal set and fixed signs on the hemispheres bounded by this equatorial sphere. By passing to a further subsequence, we may assume that the appropriate time-translation of $\mathscr{S}(r(\e_j)q)$ gives rise to a Brakke flow $\Sigma_t$ which satifies the conditions described in Section \ref{sec:main}. We can now use the continuity argument employed to characterize the limit of $\Sigma_t$ as $t\to -\infty$ to conclude that $\Sigma_t \to S^3 \cap y^\perp$ as $t \to +\infty$, in the varifold sense. The conclusion can be extended to every direction in $\mathcal{O}$ using the equivariance of the solution map and of the induced varifolds.
\end{proof}

Since the sequence $\e_j \downarrow 0$ in Proposition \ref{forward} can be extracted from any sequence of parameters converging to zero, this finishes the proof of Theorem \ref{thm2}.

\subsection{Final remarks}

We conclude with a brief remark concerning the \emph{forward limits, sweepouts, and gradient flows of $E_\e$} in $S^3$. The topological argument given in the proof of Proposition \ref{prop_parabolicflow} shows that that for every small $\e>0$ and for every direction $p\in S^4 \cap\{x_1=0\} \subset \R^5$, we can find an orbit contained in $\partial B_{r(\e)}(0)$ such that all the corresponding solutions to \eqref{PAC} converge to ground state solutions as $t \to +\infty$. This is because the geodesic in $\partial B_{r(\e)}(0)$ joining $\pm r(\e) e_1$ through $r(\e)p$ must contain a point $a$ such that $\mathscr{S}(a)$ does not converge to $\pm 1$. Since ground state solutions have Morse index $1$, it seems feasible that this is the unique such orbit that intersects this geodesic, and that altogether these orbits parametrize a $3$-sphere in the intersection of the unstable manifold of $u_\e^{-\infty}$ and the stable manifold of the set of ground state solutions. 

An interesting consequence of Proposition \ref{forward} is that one can produce a $1$-sweepouts for the energy functional $E_\e$ (in the sense of \cite{GG}), for small $\e>0$, using the gradient flow of the energy and the critical points $\pm u_{\e}^{-\infty}$ that vanish on the Clifford torus $T_c$. By studying the forward limits of all solutions $\mathscr{S}(a)$ (not only those in the orbits $\mathcal{O}_\e$), one may expect to be able to use this construction to study the second min-max critical value $c_\e(2)$ for the Allen-Cahn equation, in connection with the $2$-width of the sphere, see e.g. \cite{MNMorse}. We plan to address these questions in future work.

\bibliography{main}
\bibliographystyle{acm}

\end{document}